\documentclass[11pt,a4paper]{article}

\usepackage[dvips]{epsfig}
\usepackage{amssymb,graphics,amsmath,amsthm,amsopn,amstext,amsfonts}

\usepackage{amsmath,epsfig,epsf,psfrag}
\usepackage{a4wide,amsmath,epsfig,epsf,psfrag,amsthm, amssymb}

\newcommand{\wt}{\widetilde}
\newcommand{\wh}{\widehat}

\newcommand{\ol}{\overline}

\newtheorem{theorem}{Theorem}[section]
\newtheorem{lemma}{Lemma}[section]
\newtheorem{coro}{Corollary}[section]
\newtheorem{prop}{Proposition}[section]

\newtheorem{example}{Example}[section]
\newtheorem{remark}{Remark}[section]

\begin{document}

\title{Local Asymptotics of $P$-Spline Smoothing}

\author{Xiao  Wang,\footnote{Department of Statistics, Purdue University, West Lafayette, IN 47909,
U.S.A. Email: wangxiao@purdue.edu.} \ \ Jinglai
Shen,\footnote{Department of Mathematics and Statistics,
University of Maryland Baltimore County, Baltimore, MD 21250,
U.S.A. Email: shenj@umbc.edu.}\ \  and \ \  David
Ruppert\footnote{ School of Operations Research and Information
Engineering and Department of  Statistical Sciences, Cornell
University, Ithaca, NY 14853, U.S.A. Email: dr24@cornell.edu. }}

%
%
%
%
%
%

%
%

%

\maketitle
\begin{abstract}
This paper addresses asymptotic properties of general penalized
spline estimators with an arbitrary B-spline degree and an
arbitrary order difference penalty. The estimator is approximated
by a solution of a linear differential equation subject to
suitable boundary conditions. It is shown that, in certain sense,
the penalized smoothing corresponds approximately to smoothing by
the kernel method. The equivalent kernels for both inner points
and boundary points are obtained with the help of Green's
functions of the differential equation. Further, the asymptotic
normality is established for the estimator at interior points. It
is shown that the
 convergence rate is independent of the degree of the splines,
and the number of knots does not affect the asymptotic
distribution, provided that it tends to infinity fast enough.\\

\end{abstract}

%
%
%
%

%
\section{Introduction}

Consider the problem of estimating the function $f:\mathbb [0, 1]
\rightarrow \mathbb R$ from a univariate regression model  $y_i =
f(t_i)+\epsilon_i$, $i=1, \ldots, n$, where the $t_i$ are
pre-specified design points and the $\epsilon_i$ are iid normal
random variables with mean $0$ and variance $\sigma^2$. This paper
presents a local asymptotic theory of penalized spline estimators
of $f$.

The penalized spline regression model with difference penalty was
introduced by Eilers and Marx (1996), who coined the term
``$P$-splines'', but using less knots for the regression problem
can be traced back at least to O'Sullivan (1986). Penalized spline
smoothing  has become popular over the last decade and the uses of
low rank bases lead to highly tractable computation. The
methodology and applications of $P$-splines are discussed
extensively in Ruppert, Wand and Carroll (2003). On the other
hand, asymptotic properties of the $P$-spline estimators are less
explored in the literature. A few exceptions include recent papers
such as Hall and Opsomer (2005), Li and Ruppert (2008), and
Claeskens, Krivobokova, and Opsomer (2009). Hall and Opsomer
(2005) placed knots continuously over a design set and established
consistency of the estimator. Li and Ruppert (2008) developed an
asymptotic theory of $P$-splines for piecewise constant and linear
B-splines with the first and second order difference penalties.
Claeskens, Krivobokova, and Opsomer (2009) studied bias, variance
and asymptotic rates of the $P$-spline estimator under different
choices of the number of knots and penalty parameters. An
interested reader may also refer to Pal and Woodroofe (2007), Shen
and Wang (2009), and Wang and Shen (2009) for shape constrained
regression estimators and their applications.

The $P$-spline model approximates the regression function
by $f^{[p]}(x) = \sum_{k=1}^{K_n+p} b_k B^{[p]}_k(x)$, where
$\big\{ B^{[p]}_k : k = 1, \ldots, K_n + p \, \big\}$ is the
$p\,$th degree B-spline basis with knots $0 = \kappa_0 < \kappa_1
< \cdots < \kappa_{K_n} = 1$. The value of $K_n$ will depend upon
$n$ as discussed below. The spline coefficients
$\hat{b}=\{\hat{b}_k, k=1, \ldots, K_n+p\}$ subject to the
$m$th-order difference penalty are chosen to minimize
\begin{equation}\label{equ:p}
\sum_{i=1}^n\big[y_i - \sum_{k=1}^{K_n+p}b_k
B_{k}^{[p]}(t_i)\big]^2 + \lambda^* \sum_{k=m+1}^{K_n+p}\big[\Delta^m
(b_k)\big]^2,
\end{equation}
where $\lambda^*>0$ and $\Delta$ is the backward difference
operator, i.e., $\Delta b_k \equiv b_k - b_{k-1}$ and
\begin{equation}\label{equ:diff0}
\Delta^m b_k = \Delta \Delta^{m-1}b_k = \cdots =  \sum_{j=0}^m
(-1)^{m-j} {m\choose j}b_{k-m+j}.
\end{equation}
For simplicity, we assume that both the design points and the
knots are equally spaced on the interval $[0,1]$. We also assume
that $n/K_n$ is an integer denoted by $M_n$. Hence every $M_n$th
design point is a knot, that is, $\kappa_j = t_{jM_n}$ for $j =
1,\ldots, K_n$; a more general case is discussed briefly in
Section \ref{sec:diss}. The $P$-spline estimator is given by
$\hat{f}^{[p]}(x)=\sum_{k=1}^{K_n+p} \hat{b}_k B^{[p]}_k(x)$.

This paper develops a general asymptotic theory of $P$-splines
under an arbitrary choice of $p$ and $m$. It is shown that the
$P$-spline estimator can be approximated by the solution of an
ordinary differential equation (ODE) with suitable boundary conditions. This
estimator is then shown to be described by a kernel estimator,
using a Green's function obtained from a closely related boundary
value problem as a kernel. The asymptotic properties of the
estimator thus are explicitly established based on the Green's
function and the solution of the differential equation.
%
%
It is worth mentioning that asymptotic analysis of smoothing
splines using Green's functions was performed by Rice and
Rosenblatt (1983), Silverman (1984), Messer (1991), Nychka (1995)
and Pal and Woodroofe (2007). However, these papers only treat
limited special cases. In contrast, the current paper develops a
general framework for $P$-splines. This framework leads to a
relatively simpler approach to obtain a closed-form expression of
an equivalent kernel for both inner points and boundary points at
the first time. Further, we show that the convergence rate of
$\hat f^{[p]}$ depends only on $m$ but not on $p$, as long as
$K_n$ tends to infinity fast enough; see Corollary \ref{coro:asy}
where $K_n$ is of order $n^\gamma$, where $\gamma>(2m-1)/(4m+1)$.

The contributions of the present paper are twofold: (i) the paper
develops a general approach for asymptotic analysis of a
$P$-spline estimator with an arbitrary spline degree and arbitrary
order difference penalty via Green's functions. To handle a
general $P$-spline estimator, various techniques for linear ODEs
are exploited to obtain a corresponding Green's function. (ii) the
closed-form expressions of equivalent kernels for both inner and
boundary points are established and convergence rates are
developed for general $P$-spline estimators. Compared with the
existing results based on matrix techniques, e.g. Li and Ruppert
(2008) and Claeskens, Krivobokova, and Opsomer (2009), the use of
Green's functions considerably simplifies the development  and
yields an instrumental alternative to establish the equivalent
kernels for general $P$-splines. Moreover, this also leads to the
convergence rates and the observation that the rates are
independent of the splines' degrees and the number of knots for an
arbitrary $P$-spline estimator. While this observation is pointed
out by Li and Ruppert (2008) for piecewise constant and piecewise
linear splines and is conjectured for general $P$-splines, no
rigorous justification has been given for general $P$-splines in
the literature; the current paper offers a satisfactory answer to
this issue in a general setting.

The paper is organized as follows. Section
\ref{sec:charact_estimator} characterizes the general $P$-spline
estimator as an approximate solution of a linear differential
equation subject to suitable boundary conditions.
Section~\ref{sec:gr} investigates the solution of such the
differential equation and obtains the related Green's functions as
equivalent kernels for a $P$-spline estimator of an arbitrary
B-spline degree with any order difference penalty. Using these
Green's functions, the asymptotic properties of $P$-splines are
established in Section~\ref{sec:asy}. Section~\ref{sec:boundary}
addresses kernel approximation near the boundary of the design
set. By formulating boundary conditions as an appropriate integral
form, an explicit equivalent kernel is obtained. Finally,
extensions to unequally spaced data and multivariate $P$-splines
are discussed in Section~\ref{sec:diss}.


\section{Characterization of the estimator}
\label{sec:charact_estimator}

Let $X=[B_k(x_i)]\in \mathbb{R}^{n\times (K_n+p)}$
be the design matrix, and let $D_m \in \mathbb{R}^{(K+p-m)\times(K+p)}$ be the
$m$th-order difference matrix such that
$D_mb = [\Delta^m (b_{m+1}), \ldots,
\Delta^m (b_{K_n+p})]^T $. The optimality condition is given by
\begin{equation}\label{equ:p2}
(X^TX + \lambda^* D_m^TD_m)\hat{b} = X^Ty,
\end{equation}
where $y=(y_1, \ldots, y_n)^T$.

To characterize the $P$-spline estimator $\hat f^{[p]}$, we
introduce more notation. Define $C \in
\mathbb{R}^{(K_n+p)\times(K_n+p)}$ and  $\tilde{C}\in
\mathbb{R}^{(K_n+p)\times n}$, respectively, as
$$C = \left[\begin{array}{ccccccccc}
1 & 0 & 0  &0& \cdots & 0 & 0\\
1 & 1 & 0 &0 & \cdots & 0 & 0\\
1 & 1 & 1 & 0 & \cdots & 0 & 0 \\
  & \cdots & & & \cdots & &\\
1 & 1 &  1 &1 &\cdots & 1 &0\\
1 & 1 &  1 &1 &\cdots & 1 &1
\end{array}
 \right] ~~~\mbox{and}~~~\tilde{C} = \left[\begin{array}{cccccccccccccccccc}
0 & 0 & 0 & \cdots & 0 &0\\
 & \cdots &&\cdots &&\\
0 & 0 & 0 &  \cdots & 0 &0\\
{\bf 1}^T & 0 & 0 & \cdots & 0 &0\\
{\bf 1}^T & {\bf 1}^T & 0 &  \cdots & 0 &0\\
   & \cdots &&\cdots &&\\
{\bf 1}^T & {\bf 1}^T & {\bf 1}^T & \cdots & {\bf 1}^T &0\\
{\bf 1}^T & {\bf 1}^T & {\bf 1}^T &  \cdots & {\bf 1}^T &{\bf 1}^T
\end{array}
 \right],$$
where ${\bf 1}= [1, 1, \cdots, 1]^T\in \mathbb{R}^{M_n\times 1}$.
Since $C$ is invertible, for any $k \in \mathbb N$, (\ref{equ:p2})
is equivalent to
\begin{equation}\label{equ:p3}
\lambda^* C^k D_m^TD_m \hat{b} + C^k X^T\hat{f}  = C^k X^Ty,
\end{equation}
where $\hat{f} = [\hat{f}^{[p]}(x_1), \ldots,
 \hat{f}^{[p]}(x_n)]^T$ and $C^k = \underbrace{C C\cdots C}_{k-\mbox{copies}}$.
The matrix $D_m^TD_m$ is a banded symmetric matrix. Except for the
first $m$ and last $m$ rows, every row of $D_m^TD_m$ has the form $(0,
\cdots, 0, \omega_0^*, \omega_1^*, \cdots, \omega_{2m}^*, 0,
\cdots, 0)$, where
$\omega_j^* = (-1)^m (-1)^{2m-j} {2m\choose j}$, $j=0, \ldots, 2m$.
Moreover, except for the first $m-k$ and last $m$ rows, the $i$th
row of $C^k D_m^TD_m$ has the form
$$\Big( \, \underbrace{0, \cdots, 0,}_{(i-m+k-1)-\mbox{copies}} \omega_0, \cdots, \omega_{2m-k}, \underbrace{0, \cdots, 0}_{(K_n+p)-(i+m)-\mbox{copies}}\, \Big),$$
where \begin{equation} \omega_j = (-1)^m (-1)^{2m-k-j}
{2m-k\choose j}, \ j=0, \ldots, 2m-k.
\end{equation} Further, the elements of the last $k$ rows of $C^k D_m^TD_m$ are all zeros.
In particular, when $k=m$,
\begin{equation}\label{equ:dd}
C^m D_m^TD_m \hat{b} = (-1)^m\
\big[~\Delta^{m}\hat b_{m+1}, \Delta^{m}\hat b_{m+2}, \cdots,
\Delta^{m}\hat b_{K_n+p}, 0, \ldots, 0~\big]^T.
\end{equation}
It is also interesting to note the derivative formula for B-spline functions (de Boor, 2001)
\begin{equation}\label{equ:derivative}{d^{\,l}\over dx^l}\sum_{k=1}^{K_n+p}b_kB_k^{[p]}(x)
= \sum_{k=l+1}^{K_n+p}K_n^l\Delta^l b_k \
B_{k-1}^{[p-l]}(x),~~~ l\le p.\end{equation}
Hence,
$${d^m\over dx^m}\sum_{k=1}^{K_n+m} \hat{b}_k B_k^{[m]}(x) =
K_n^m \sum_{k=m+1}^{K_n+m}\Delta^m \hat{b}_kB_{k-m}^{[0]}(x),$$
and therefore,
\begin{equation}\label{equ:diff2}\Delta^m \hat{b}_{m+k} = {1\over K_n^m}{d^m\over dx^m}\hat{f}^{[m]}(x), ~~~~ x\in
(\kappa_{k-1}, \kappa_k], ~~~k=1, \ldots, K_n.\end{equation}

Let $\omega_1$ be the uniform distribution on $x_1, \ldots, x_n$
and $\omega_2$ be the uniform distribution on $\kappa_1, \ldots,
\kappa_{K_n}$. Let $g$ and $\check{f}$ be two piecewise constant
functions for which $g(x_k) = y_k$ and
$\check{f}(x_k)=\hat{f}(x_k)$ for $k = 1, . . . , n$,
respectively. Let $G_1(x) = \int_0^x g(t)d\omega_1(t)$,
$\check{F}_1(x)=\int_0^x \check{f}(t)d\omega_1(t)$, $\hat{F}_1(x) = \int_0^x \hat{f}(t)dt$, and for $k\ge
2$, define
$$G_k(x) = \int_0^x G_{k-1}(t)d\omega_2(t), ~~~~
\check{F}_k(x) = \int_0^x\check{F}_{k-1}(t)d\omega_2(t), ~~~~\hat{F}_k(x) =
\int_0^x \hat{F}_{k-1}(t)dt.$$ To obtain the analogous
representation for $\hat f$, we introduce a few variables and
functions related to the true regression function $f$. Define
$\Phi_1(x) = \int_0^x f(t)dt$, $\tilde \Phi_1(x) = \int_0^x
f(t)d\omega_1(t)$, and for $k\ge 2$,
$$\Phi_k(x)=\int_0^x\Phi_{k-1}(t)dt, ~~~~\tilde
\Phi_k(x)=\int_0^x\tilde\Phi_{k-1}(t)d\omega_2(t).$$
Letting $R = C\ X^T - \tilde{C}$, we have
$C^m X^T\hat{f}=C^{m-1} \tilde{C}\hat{f} +C^{m-1} R\hat{f}$.
Therefore, the $j$th row of (\ref{equ:p3}), when $k=m$, can
be written as
\begin{equation}\label{equ:diff}
\check{F}_m(\kappa_{j+p-1}) + R_{fj} + (-1)^m {\lambda^*\over nK_n^{m-1}} \Delta^m
b_{m+j}  = G_m(\kappa_{j+p-1}) + R_{yj}, \ \ j=1,\ldots,
K_n,
\end{equation}
where $R_{fj}$ and $R_{yj}$ are the $j$th row of ${1\over
  nK_n^{m-1}}C^{m-1} R\hat{f}$ and ${1\over
  nK_n^{m-1}}C^{m-1} Ry$, respectively.
Furthermore, since the elements of the last $k$ rows of $C^k D_m^TD_m$ are all zeros,
we also have
  \begin{equation}\label{equ:bd1}
\check{F}_k(1) =
G_k(1), ~~~~~ k=1, \ldots, m.
\end{equation}

Next, we proceed by replacing that difference equation
(\ref{equ:diff}) by an analogous differential equation. We shall
focus on the case when $p=m$ first; the case when $p\neq m$ will
be discussed in Section~\ref{sec:asy}. For any $x\in [0,1]$,
letting $k_x=\lfloor K_nx \rfloor+1$, (\ref{equ:diff}) gives
\begin{equation}\label{equ:diffx}
\check{F}_m(\kappa_{k_x+p}) + R_{f,k_x+1} + (-1)^m {\lambda^*\over nK_n^{m-1}} \Delta^m
b_{m+k_x+1}  = G_m(\kappa_{k_x+p}) + R_{y,k_x+1}.\end{equation}
Define \begin{equation}
\tilde R(x) =
\hat{F}_m(x)-G_m(x)+G_m(\kappa_{k_x+p})-\check{F}_m(\kappa_{k_x+p})
+R_{y,k_x+1}-R_{f,k_x+1}.
\end{equation}
Then, from (\ref{equ:diff2}) and (\ref{equ:diffx}), $\hat{F}_m$ solves the ordinary differential equation
\begin{equation}\label{equ:ode}
(-1)^m \alpha \hat{F}^{(2m)}_m(x) + \hat{F}_m(x) = G_m(x) + \tilde R(x), ~~~~ 0\le x\le 1,
\end{equation}
where $\alpha = \lambda^*/(nK_n^{2m-1})$. We have $2m$ boundary
conditions for (\ref{equ:ode}):
$$\hat{F}_m^{(k)}(0)=0, ~~~\hat{F}_m^{(k)}(1) = G_{m-k}(1)+e_{m-k}, ~~~ k=0, \ldots, m-1,$$ where $e_{m-k} =
\hat{F}_m^{(k)}(1)-\check{F}_{m-k}(1)$. We shall show that $\hat{f}^{[p]}$  is
stochastically bounded, therefore the $e_k$ are
small with an order of $O_p(1/n)$.

%

\section{Green's functions}\label{sec:gr}

The solution to (\ref{equ:ode}) can be represented by a
corresponding Green's function explicitly. It shall be shown that
the $P$-spline estimator can be approximated by a kernel
estimator, using the corresponding Green's function.
For this end, consider the differential equation
\begin{equation}\label{eqn:govern}
(-1)^m \alpha F^{(2m)}(t) + F(t) = G(t),~~~~~~~ 0\le t\le 1,
\end{equation}
subject to the boundary conditions $F^{(i)}(0)=0$ and
$F^{(i)}(1)=G^{(i)}(1)$, $i = 0,\ldots, m-1$. Let $\beta \equiv
\alpha^{-1/(2m)}$. We consider two cases: (1) $m$ is even; and (2)
$m$ is odd.

%
\subsection{Even $m$} \label{sect:kernel_even}

In this case, the characteristic equation is given by
$\lambda^{2m} + \beta^{2m} =0$, and we obtain $2m$ eigenvalues
\[
   \lambda_k = \beta \Big[ \cos\frac{(1+2k)\pi}{2m} + \imath \,
     \sin \frac{(1+2k)\pi}{2m} \Big], \ \ \ k=0, 1, \cdots, 2m-1.
\]
Let
$$\mu_k = \cos\frac{(1+2k)\pi}{2m}  ~~\mbox{ and }~~  \omega_k =  \sin
\frac{(1+2k)\pi}{2m}.$$ Then the homogeneous ODE: $\alpha
F^{(2m)}(t) + F(t) =0$ has $2m$ solutions
\[
  e^{(\pm \mu_k \pm \imath \omega_k) \beta t} = e^{ \pm \beta \mu_k t } \, \big[ \cos (\beta\omega_k t)  \pm \imath \,
   \sin(\beta\omega_k t) \big], \ \ \ k=0, \cdots, \frac{m}{2}-1,
\]
where $\mu_k>0$ and $\omega_k>0$ for $k=0, \cdots, \frac{m}{2}-1$.

To find the corresponding Green's function for the ODE: $\alpha
F^{(2m)}(t) + F(t) = G(t)$ on $[0, 1]$, we define the following
function
\begin{equation} \label{eqn:L_even}
   L(t) \, \equiv \, \sum^{\frac{m}{2}-1}_{k=0} \, \beta \, e^{-\beta \mu_k t} \big[
    \, c_k  \,\cos (\omega_k \beta t) +  d_k \, \sin(\omega_k \beta t) \big],
\end{equation}
where the coefficients $c_k, d_k$ are to be determined, and
$
    K(t, s) \, \equiv \, L(|t-s|).
$ Since $L$ is a linear combination of the solutions of the
homogeneous ODE, $L^{(2m)} + \beta^{2m} L=0$ also holds. Let
\[
   F_0(t) \equiv \int^1_{0} K(t, s) G(s) ds, \ \ \ t\in[0, 1].
\]
By noting $F_0(t) = \int^t_0 L(t-s) G(s) ds + \int^1_t
L(s-t)G(s)ds$ for all $ t \in [0, 1]$, it is easy to verify that
if
\begin{equation} \label{eqn:L}
   L^{(k)}(t)\big|_{t=0} = 0, \ \ \forall \ \ k=1, 3, \cdots, 2m-3,
   \ \ \mbox{ and } \ \ \ L^{(2m-1)}(t)\big|_{t=0} = \frac{ \beta^{\, 2m} }{2},
\end{equation}
then $F_0(t)$ is a solution of $\alpha F^{(2m)} + F =
G$.

To find the coefficients $c_k, d_k$, define
\[
   p_{\, k}(t) \, \equiv \, e^{-\beta \mu_k t} \big[
    \, c_k  \,\cos (\omega_k \beta t) +  d_k \, \sin(\omega_k \beta t)
    \big], \ \ \
   q_k(t) \, \equiv \, e^{-\beta \mu_k t} \big[
    \, - c_k  \,\sin (\omega_k \beta t) +  d_k \, \cos(\omega_k \beta t)
    \big].
\]
Hence $p_{\, k}(0) = c_k$ and $q_k(0) = d_k $. Since
\begin{equation} \label{eqn:p_q}
   \begin{pmatrix}  p'_{\, k}(t)  \\ q'_k(t) \end{pmatrix} \, =
   \,  \beta \underbrace{ \begin{bmatrix} -\mu_k & \omega_k \\ -\omega_k &
   -\mu_k \end{bmatrix} }_{A_k}
     \begin{pmatrix}  p_{\, k}(t)  \\ q_k(t) \end{pmatrix},
\end{equation}
we have
\[
   \begin{pmatrix}  p^{(j)}_{\, k}(t)  \\ q^{(j)}_k(t) \end{pmatrix} \, =
   \,  \big( \beta A_k \big)^j
     \begin{pmatrix}  p_{\, k}(t)  \\ q_k(t) \end{pmatrix},
\]
where $p^{(j)}_{\, k}(t)$ and $q^{(j)}_k(t)$ stand for the $j$-th
derivatives of $p_k$ and $q_k$ respectively. Letting $A^j_k(i,
\ell)$ denote the $(i,\ell)$-element of $A^j_k$, we obtain the
following linear equation for $\{ c_k, d_k \}$ from (\ref{eqn:L}):
\begin{equation} \label{eqn:coefficient_even}
  {\small \underbrace{
   \begin{bmatrix}
      A_0(1,1) & A_0(1,2) &  \cdots & \cdots
      & A_{\frac{m}{2}-1}(1,1) & A_{\frac{m}{2}-1}(1,2) \\
      A^3_0(1,1) & A^3_0(1,2) &  \cdots & \cdots
      & A^3_{\frac{m}{2}-1}(1,1) & A^3_{\frac{m}{2}-1}(1,2) \\
       \vdots & \vdots &  & & \vdots & \vdots \\
     A^{(2m-3)}_0(1,1) & A^{(2m-3)}_0(1,2) &  \cdots & \cdots
      & A^{(2m-3)}_{\frac{m}{2}-1}(1,1) & A^{(2m-3)}_{\frac{m}{2}-1}(1,2) \\
    A^{(2m-1)}_0(1,1) & A^{(2m-1)}_0(1,2) &  \cdots & \cdots
      & A^{(2m-1)}_{\frac{m}{2}-1}(1,1) & A^{(2m-1)}_{\frac{m}{2}-1}(1,2)
   \end{bmatrix}
    }_{A^e}
   \begin{bmatrix} c_0 \\ d_0 \\ \vdots \\ c_{\frac{m}{2}-1} \\ d_{\frac{m}{2}-1}
   \end{bmatrix} \, = \,
   \begin{bmatrix}
       0 \\ 0 \\ \vdots \\ 0 \\ \frac{1}{2}
   \end{bmatrix}.
   }
\end{equation}
It shall be shown in Lemma~\ref{lem:coefficient} that the above
equation has a unique solution.

%
\subsection{Odd $m$ } \label{sect:kernel_odd}

The characteristic equation is given by $\lambda^{2m} - \beta^{2m}
=0$ and the eigenvalues are:
\[
   \lambda_k = \beta \Big( \cos\frac{k\pi}{m} + \imath \,
     \sin \frac{k\pi}{m} \Big), \ \ \ k=0, 1, \cdots, 2m-1.
\]
Then the homogeneous ODE: $\alpha F^{(2m)}(t) + F(t) =0$ has $2m$
solutions: $e^{\pm \beta t}$ and
\[
   e^{(\pm \mu_k \pm \imath \omega_k) \beta t} = e^{ \pm \beta \mu_k t } \, \big[ \cos (\beta\omega_k t)  \pm \imath \,
   \sin(\beta\omega_k t) \big], \ \ \ k=1, \cdots, \frac{m-1}{2},
\]
where $\mu_k = \cos\frac{k \pi}{m}>0$ and $\omega_k=\sin\frac{k
\pi}{m}>0$ for $k=1, \cdots, \frac{m-1}{2}$. Similar to the even
case,  define
\begin{equation} \label{eqn:P_odd}
   P(t) \, \equiv \, c_0 \, \beta e^{-\beta t} + \sum^{(m-1)/2}_{k=1} \, \beta \, e^{-\beta \mu_k t} \big[
    \, c_k  \,\cos (\omega_k \beta t) +  d_k \, \sin(\omega_k \beta t) \big],
\end{equation}
where the coefficients $c_k, d_k$ are to be determined, and $P(t)$
satisfies $P^{(2m)}(t) - \beta^{2m} P(t)=0$. Let $K(t,s) \equiv
P(|t-s|)$ and $   F_0(t) \equiv \int^1_{0} K(t,s) G(s) ds$. It can
be verified that if
\begin{equation} \label{eqn:P}
   P^{(k)}(t)\big|_{t=0} = 0, \ \ \forall \ \ k=1, 3, \cdots, 2m-3,
   \ \ \mbox{ and } \ \ \ P^{(2m-1)}(t)\big|_{t=0} = -\frac{\beta^{\,
   2m}}{2},
\end{equation}
then $F_0(t)$ is a solution of $\alpha F^{(2m)} - F =
-G$. Similarly, it can be shown that $P$ is also a $2m$th-order kernel.
To find the coefficients $c_0$ and $c_k, d_k$, we may use $p_{\,
k}$, $q_k$ and $A_k$ introduced in the last subsection. Indeed, we
obtain the following linear equation for $c_0$ and $\{ c_k, d_k
\}$ from (\ref{eqn:P}):
\begin{equation} \label{eqn:coefficient_odd}
  {\small
  \underbrace{ \begin{bmatrix}
     -1 & A_1(1,1) & A_1(1,2)   & \cdots & \cdots & A_{\frac{m-1}{2}}(1,1)
         & A_{\frac{m-1}{2}}(1,2) \\
     -1 & A^3_1(1,1) & A^3_1(1,2)  & \cdots & \cdots
      & A^3_{\frac{m-1}{2}}(1,1) & A^3_{\frac{m-1}{2}}(1,2) \\
       \vdots & \vdots & \vdots &   & & \vdots & \vdots  \\
     -1 & A^{(2m-3)}_1(1,1) & A^{(2m-3)}_1(1,2) & \cdots & \cdots
      & A^{(2m-3)}_{\frac{m-1}{2}}(1,1) & A^{(2m-3)}_{\frac{m-1}{2}}(1,2) \\
    -1 & A^{(2m-1)}_1(1,1) & A^{(2m-1)}_1(1,2) & \cdots & \cdots
      & A^{(2m-1)}_{\frac{m-1}{2}}(1,1) & A^{(2m-1)}_{\frac{m-1}{2}}(1,2)
   \end{bmatrix}
   }_{A^o}
   \begin{bmatrix} c_0 \\ c_1 \\ d_1 \\ \vdots \\ c_{\frac{m-1}{2}} \\ d_{\frac{m-1}{2}}
   \end{bmatrix} \, = \,
   \begin{bmatrix}
      0 \\ 0 \\ 0 \\ \vdots \\ 0 \\ -\frac{1}{2}
   \end{bmatrix}.
   }
\end{equation}

%
\subsection{The equivalent kernels}

\begin{lemma} \label{lem:coefficient}
  Each of the equations
  (\ref{eqn:coefficient_even}) and (\ref{eqn:coefficient_odd}) has a
  unique solution.
\end{lemma}
\begin{proof}We introduce some trigonometric identities
to be used in the proof. Let $p, q \in \mathbb N$. By observing
$\sin(-\theta)\sum^p_{k=1} \cos[(2k-1)\theta] = \frac{1}{2}
\sum^p_{k=1} \big[ \sin(2(k-1)\theta) +\sin(-2k\theta)\big] $ and
$\sin(\theta)\sum^p_{k=1} \sin[(2k-1)\theta] = \frac{1}{2}
\sum^p_{k=1} \big[ \cos(2(k-1)\theta) -\cos(2k\theta)\big] $, it
is easy to see  (i) for $\displaystyle \theta=\frac{q}{2p}\pi$,
$\sum^p_{k=1} \cos[(2k-1)\theta] = 0$; and (ii) for $\displaystyle
\theta=\frac{q}{p}\pi$, $\sum^p_{k=1} \sin[(2k-1)\theta] = 0$.

We consider an even $m$ first. Let $\displaystyle \theta \equiv
\pi -\frac{\pi}{2m} $. It is clear that $-\mu_k = \cos\big( (2k+1)
\theta \big)$ and $\omega_k = \sin\big( (2k+1) \theta \big)$ for
all $k=0, \cdots, m/2-1$. Hence $A_k$ in (\ref{eqn:p_q}) becomes
$A_k = M( (2k+1)\theta) $, where $M(\cdot) \in \mbox{SO}(2)$ is
given by
\begin{equation} \label{eqn:M_matrix}
 M (\cdot) \equiv
   \begin{bmatrix} \cos(\cdot) & \sin(\cdot) \\
    -\sin(\cdot) & \cos (\cdot) \end{bmatrix}.
\end{equation}
Thus $(A_k)^j = M( j (2k+1)\theta)$. Let $A^e_{i\bullet }$ denote
the $i$th row of $A^e$ and $\eta_i \equiv (2i-1) \theta$. Hence,
\begin{eqnarray*}
  A^e_{i \bullet} \,  = \,    \begin{pmatrix} \ \cos (\eta_i ) &  \sin(\eta_i) &
  \cos(3\eta_i ) &  \sin(  3 \eta_i ) &
   \cdots & \cos(  (m-1) \eta_i ) & \sin(  (m-1)\eta_i ) \
   \end{pmatrix}.
\end{eqnarray*}
Therefore, $  A^e_{i \bullet } \, \big( A^e_{i\bullet } \big)^T =
\frac{m}{2}$, and if $i \ne j$, then
\begin{eqnarray*}
    A^e_{i \bullet } \,  \big(  A^e_{j\bullet} \big)^T  & = &
   \sum^{\frac{m}{2}}_{\ell=1} \Big[ \cos ((2\ell-1)(2i-1)\theta) \, \cos
   ((2\ell-1)(2j-1)\theta) \\
  & & \ \ \ \ \ \ \ \ \ \ \ + \sin ((2\ell-1)(2i-1)\theta) \, \sin ((2\ell-1)(2j-1)\theta)
   \Big] \\
   & = & \sum^{\frac{m}{2}}_{\ell=1} \cos\big( 2(2\ell-1)(i-j)\theta \big)
    \, = \, \sum^{\frac{m}{2}}_{\ell=1} \cos\big( (2\ell-1)(i-j)\frac{\pi}{m}
   \big) \, = \, 0,
\end{eqnarray*}
where the last step is attained from (i). This shows that $ A^e
(A^e)^T = \frac{m}{2} I$. Thus $A^e$ is invertible so that
equation (\ref{eqn:coefficient_even}) has a unique solution.

We then consider an odd $m$. In this case, $-\mu_k = \cos\big(\pi-
\frac{k\pi}{m} \big)$ and $\omega_k = \sin\big( \pi-
\frac{k\pi}{m} \big)$ for $k=1, \cdots, (m-1)/2$. Let $\gamma_k
\equiv \pi - \frac{k\pi}{m}$. Then the $i$th row of $A^o$ is given
by
\begin{eqnarray*}
  A^o_{i \bullet} &  = &   \left(  \begin{array}{cccccccccccc}
     \cos((2i-1)\pi) & \cos ((2i-1)\gamma_1 ) &  \sin((2i-1)\gamma_1) &
  \cos((2i-1)\gamma_2 ) &  \sin(  (2i-1) \gamma_2 )
   \\ \end{array} \right. \\
    &  &  \left.  \begin{array}{cccccccccccc}  \hspace{1.8in} \cdots & \cdots & \cos\big(  (2i-1) \gamma_{\frac{m-1}{2}} \big) &
      \sin\big(  (2i-1)\gamma_{\frac{m-1}{2}} \big) \end{array}
      \right).
\end{eqnarray*}
Let $A^o_{\bullet i}$ denote the $i$th column of $A^o$. Clearly $
\big( A^o_{\bullet i} \big)^T \, A^o_{\bullet i }
>0$. For $i \ne j$, either $\big( A^o_{\bullet i} \big)^T \,
A^o_{\bullet j } = \sum^m_{k=1} \cos((2k-1)\gamma_s)
\cos((2k-1)\gamma_t)$ with $s \ne t$ or $\big( A^o_{\bullet i}
\big)^T \, A^o_{\bullet j } = \sum^m_{k=1} \cos((2k-1)\gamma_s)
\sin((2k-1)\gamma_t)$,  for some $s, t \in\{1, \cdots,
\frac{m-1}{2} \}$. Since
\begin{eqnarray*}
 \sum^m_{k=1} \cos\big((2k-1)\gamma_s\big) \cos\big((2k-1)\gamma_t \big) =
  \frac{1}{2} \sum^m_{k=1} \Big[ \cos((2k-1)(\gamma_s + \gamma_t))
  +\cos((2k-1)(\gamma_s - \gamma_t))\Big], \\
 \sum^m_{k=1} \cos\big((2k-1)\gamma_s \big) \sin\big((2k-1)\gamma_t\big)=
  \frac{1}{2} \sum^m_{k=1} \Big[ \sin((2k-1)(\gamma_s + \gamma_t))
  +\sin((2k-1)(\gamma_s - \gamma_t))\Big],
\end{eqnarray*}
we conclude that $\big( A^o_{\bullet i} \big)^T \, A^o_{\bullet j
} =0$ by using (i)--(ii) established at the beginning of the
proof. This shows that $(A^o)^T A^o$ is a diagonal matrix with
positive diagonal entries. Therefore $A^o$ is invertible and
equation (\ref{eqn:coefficient_odd}) has a unique
solution.
\end{proof}

The following proposition show that $L$ and $P$ derived above
yield the equivalent kernels.

\begin{prop}
When $\beta=1$, $L(|t|)$ in (\ref{eqn:L_even}) and $P(|t|)$ in
(\ref{eqn:P_odd}) are $2m$th order kernels respectively.
\end{prop}

\begin{proof}
We consider $L(|t|)$ only since the other case follows from the
similar argument. We shall show that $\int^\infty_{-\infty}
L(|\tau|) d\tau=1$ and $\int^\infty_{-\infty} \tau^{k} L(|\tau|)
d\tau=0$ for all $k=1, \cdots, 2m-1$. This holds true trivially
when $k$ is odd. For an even $k$, by observing $ L^{(2m)}+
\beta^{2m} L =0$ (with $\beta=1$), we have
\[
  \int^\infty_{-\infty} \tau^{k} L(|\tau|) d\tau= 2\int^\infty_{0}
  \tau^{k} L(\tau) d\tau = -2\int^\infty_{0} \tau^{k} L^{(2m)}(\tau) d\tau.
\]
Repeatedly using the integration by part, we deduce
\[
   \int^t_0 \tau^k L^{(2m)}(\tau) d\tau = \sum^k_{i=0}
   \frac{k!}{(k-i)!} (-1)^i t^{k-i}  \Big( L^{(2m-1-i)}(t) -
   L^{(2m-1-i)}(0) \Big).
\]
In light of (\ref{eqn:L}), we obtain the desired result.
\end{proof}

\begin{figure}[h]
    \begin{center}
        \epsfig{figure=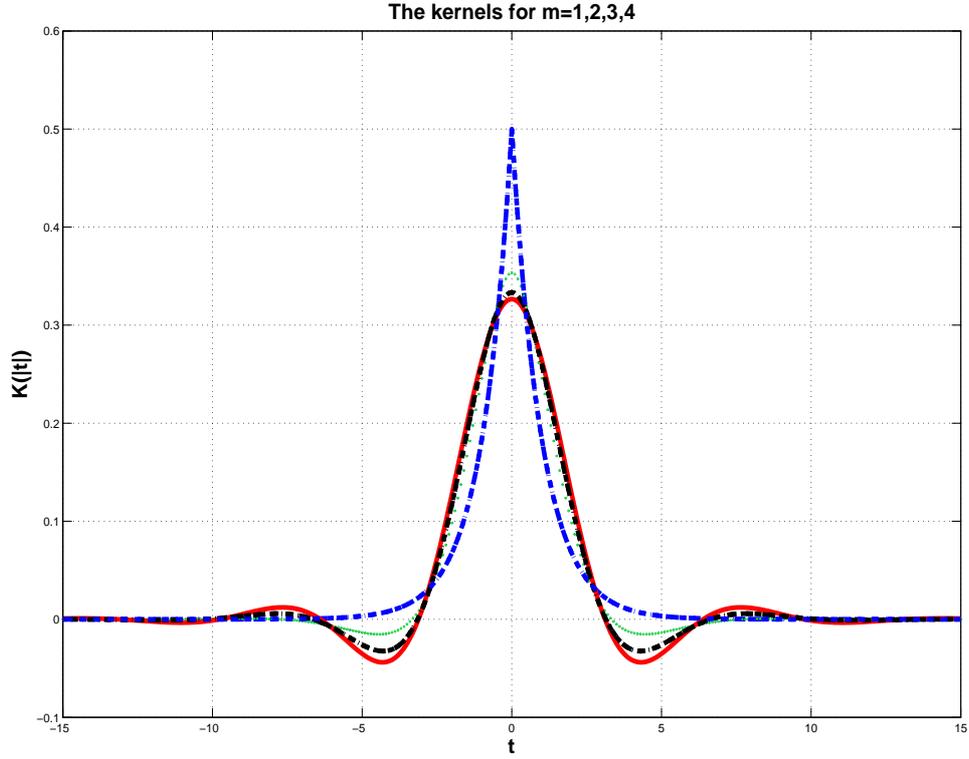, width=6.0 in, height =4.4 in, }
        \caption{The equivalent kernels for $m=1, 2, 3, 4$: (i) $m=1$: the dashed line; (ii) $m=2$: the dotted line;
         (iii) $m=3$: the dashdot line; (iv) $m=4$: the solid line.}
        \label{fig:kernels}
      \end{center}
\end{figure}

\begin{example} \label{example:kernel} \rm
As an illustration, the closed-form expressions of the first four
equivalent kernels are given below and their plots  are shown in
Figure~\ref{fig:kernels}, respectively.
\begin{eqnarray*}
   m=1: \ \ K(t)  & = & \frac{1}{2} e^{-|t|} \\
   m=2: \ \ K(t)  & = & \frac{1}{2\sqrt{2}} \, e^{-\frac{1}{\sqrt{2}} |t|}
    \Big( \, \cos\frac{|t|}{\sqrt{2}} + \sin \frac{|t|}{\sqrt{2}} \, \Big) \\
   m=3: \ \ K(t)  & = & \frac{1}{6} e^{-|t|} +  \, e^{-\frac{1}{2} |t|}
    \Big( \, \frac{1}{6} \cos\frac{\sqrt{3}|t|}{2} + \frac{\sqrt{3}}{6} \sin\frac{\sqrt{3}|t|}{2}
     \, \Big) \\
   m=4: \ \ K(t)  & = &  e^{-0.9239 |t|}
    \Big( \, 0.2310 \cos(0.3827|t|) + 0.0957 \sin(0.3827|t|)
     \, \Big)  \\
   & &  \, + \, e^{-0.3827 |t|} \Big( \, 0.0957  \cos(0.9239|t|) + 0.2310 \sin(0.9239|t|)
     \, \Big)
\end{eqnarray*}

\end{example}

%
\subsection{Boundary conditions} \label{sect:boundary_conditions}

Recall that the boundary conditions for the ODE~(\ref{eqn:govern})
are $F^{(i)}(0) = 0$, $F^{(i)}(1) = G^{(i)}(1)$, $i=0, \cdots,
m-1$. In the following, we consider an even $m$ first. In this
case,  the homogeneous ODE: $ F^{(2m)} + \beta^{2m} F =0$ has the
following $2m$ (linearly independent) solutions:
\[
  e^{-\beta\mu_k t} \cos(\beta\omega_k t), \ \ \ e^{-\beta\mu_k t} \sin(\beta\omega_k
  t), \ \ \ \ e^{-\beta\mu_k(1- t)} \cos(\beta\omega_k t), \ \ \ e^{-\beta\mu_k (1-t)}
  \sin(\beta\omega_k t),
\]
where $k=0, \cdots, \frac{m}{2}-1$ and $\mu_k, \omega_k>0$ for the
above $k$. The solution to ODE (\ref{eqn:govern}) subject to the
boundary conditions can be written as
\begin{equation}\label{eqn:F_t}
   F(t) = \underbrace{ \int^1_{0} L(|t-s|) G(s) ds }_{F_0(t)} + J(t),
\end{equation}
where
\begin{equation}\label{eqn:J_t1}
J(t) =  \sum^{\frac{m}{2}-1}_{k=0} \Big\{  e^{-\beta\mu_k t}\big[ a_k
    \cos(\beta\omega_k t) + b_k \sin(\beta\omega_k t) \big] + e^{-\beta\mu_k(1- t)}
   \big[ a^+_k \cos(\beta\omega_k t) +  b^+_k \sin(\beta\omega_k
   t) \big] \Big\},
\end{equation}
and the coefficients $a_k, b_k, a^+_k, b^+_k$ are to be
determined from the boundary conditions, and the kernel $L$ is
given in (\ref{eqn:L_even}). Define $\displaystyle \| G\| \equiv
\sup_{t\in[0, 1]} | G(t)|$. Let ${\bf G}=\big(\| G \|, G(1),
G'(1), \cdots, G^{(m-1)}(1) \big)$, and \begin{equation}\label{eqn:a}{\bf a}\, = \, \Big( \,
a_0, b_0, \cdots, a_{\frac{m}{2}-1}, b_{\frac{m}{2}-1}, a^+_0,
b^+_0, \cdots, a^+_{\frac{m}{2}-1}, b^+_{\frac{m}{2}-1} \,
\Big)^T\end{equation} be the coefficient vector.

By making use of the boundary conditions, we obtain the linear
equation $B^e {\bf a} = \bf v$, where ${\bf v}^T = [{\bf v_0},
{\bf v_1}]$,
$${\bf v_0} = \Big[\,  -F_0(0), \, -\frac{F'_0(0)}{\beta}, \, \cdots,
     \,-\frac{F^{(m-1)}_0(0)}{\beta^{m-1}} \, \Big],$$
$${\bf v_1} =\Big[ \, -F_0(1)+G(1), \, \frac{-F'_0(1) + G'(1)}{\beta}, \, \cdots,
      \, \frac{-F^{(m-1)}_0(1) + G^{(m-1)}(1)}{\beta^{m-1}} \, \Big ],$$
and
\[
  B^e = \begin{bmatrix} B^e_{11} & B^e_{12} \\  B^e_{21} &
     B^e_{22} \end{bmatrix}.
\]
Here the matrix blocks $B^e_{ij} \in \mathbb R^{m\times m}$ are
obtained via the similar technique in
Section~\ref{sect:kernel_even} as
\begin{eqnarray} \label{eqn:B_11}
  \lefteqn{ B^e_{11} = } \nonumber \\
 & {\small \begin{bmatrix} 1 & 0 &  1 & 0 & \cdots & \cdots & 1 & 0 \\
               \cos(\eta_1) & \sin(\eta_1) &  \cos(3\eta_1) & \sin(3\eta_1) & \cdots & \cdots & \cos((m-1)\eta_1) &
               \sin((m-1)\eta_1) \\
                 \vdots &  \vdots & \vdots & \vdots & & &\vdots  & \vdots \\
               \cos(\eta_{m-1}) & \sin(\eta_{m-1}) &  \cos(3\eta_{m-1}) & \sin(3\eta_{m-1}) & \cdots & \cdots & \cos((m-1)\eta_{m-1}) &
               \sin((m-1)\eta_{m-1})
              \end{bmatrix}, } \nonumber \\
\end{eqnarray}
where $\eta_k = k (\pi-\frac{\pi}{2m}), \, k=1, \cdots, m-1$, and
\begin{eqnarray*}
 \lefteqn{ B^e_{22} \, = \,  }  \\
  & {\small \begin{bmatrix}  \cos(\psi_{0,0}) & \sin(\psi_{0, 0}) & \cos(\psi_{1,0})
     & \sin(\psi_{1,0}) & \cdots & \cdots
       &   \cos(\psi_{{m\over 2}-1, 0}) & \sin(\psi_{ {m\over 2}-1, 0}) \\
       \cos(\psi_{0,1}) & \sin(\psi_{0,1}) & \cos(\psi_{1,1})
    & \sin(\psi_{1,1}) & \cdots & \cdots
       &   \cos(\psi_{ {m\over 2}-1, 1}) & \sin(\psi_{{m\over 2}-1, 1}) \\
          \vdots & \vdots &\vdots &  \vdots  & &  & \vdots & \vdots \\
       \cos(\psi_{0, m-1}) & \sin (\psi_{0, m-1}) & \cos(\psi_{1, m-1})
    & \sin(\psi_{1,m-1}) & \cdots & \cdots
       &   \cos(\psi_{{m\over 2}-1, m-1}) & \sin (\psi_{{m\over 2}-1, m-1})
 \end{bmatrix} },
\end{eqnarray*}
where $\eta^+_k =  \frac{k\pi}{2m}$ for $k=0, \cdots, m-1$, and
$\psi_{j, \ell} = \beta\, \omega_j + (2j+1)\eta^+_{\ell}$ for all
$j=0, \cdots, \frac{m}{2}, \, \ell=0, \cdots, m-1 $, and each
entry of $B^e_{12}$ and $B^e_{21}$ is of order $O(e^{-\beta})$.

\begin{lemma} \label{lem:banded_equation_even}
  Given an even $m$.
  There exist positive real numbers $\beta_*$ and $\varrho$, dependent on $m$ only, such
  that for all $\beta \ge \beta_*$, the coefficient vector $\bf a$ is unique
  and satisfies $\| \bf a \| \le \varrho \, \| \bf G \|$.
\end{lemma}

\begin{proof}
Note that for $\beta$ sufficiently large, each element of
$B^e_{12}$ and $B^e_{21}$ is sufficiently small. Hence it suffices
to show that $B^e_{11}$ and $B^e_{22}$ are invertible. For this
end, let $B^e_{11}(k)$ denote the $k$th column of $B^e_{11}$.
Define $C_{11} \equiv \big[ B^e_{11}(2) \ B^e_{11}(1) \
B^e_{11}(4) \ B^e_{11}(3) \ \cdots \ B^e_{11}(m) \ B^e_{11}(m-1)
\big]$. Letting $\vartheta=\frac{\pi}{2m}$, it can be verified
that
\begin{eqnarray*}
 \lefteqn{ B^e_{11} +\imath C_{11} \,  = \,  } \\
 & {\footnotesize   \begin{bmatrix} 1 & \imath &  1 & \imath & \cdots & \cdots & 1 & \imath \\
     -e^{-\imath \vartheta} & -\imath e^{\imath \vartheta} & -e^{-\imath 3\vartheta} & -\imath e^{\imath 3\vartheta} &
         \cdots & \cdots & -e^{-\imath (m-1)\vartheta}
      & - \imath  e^{\imath (m-1)\vartheta} \\
      \vdots& \vdots & \vdots & \vdots &   & &  \vdots & \vdots \\
        (-e^{-\imath\vartheta})^{m-1} &  \imath (- e^{\imath \vartheta})^{m-1}
        & (-e^{-\imath3\vartheta})^{m-1} &  \imath (- e^{\imath 3\vartheta})^{m-1} &
\cdots & \cdots &
      (-e^{-\imath (m-1)\vartheta})^{m-1}
      &  \imath (- e^{\imath (m-1)\vartheta})^{m-1}
      \end{bmatrix}. }
\end{eqnarray*}
Therefore $B^e_{11} +\imath C_{11}$ can be written as $
\mbox{diag}(1, \imath, 1, \imath, \cdots, 1, \imath) V$, where $V$
is an invertible Vandermonde matrix. This implies that $B^e_{11}
+\imath C_{11}$ is invertible. On the other hand, by noting
$C_{11} = B^e_{11} J$, where $J=\mbox{diag}\underbrace{ (J_*,
\cdots, J_*)}_{\frac{m}{2}-\mbox{copies}}$ with $J_*=
\begin{bmatrix} 0 & 1
\\ 1 & 0 \end{bmatrix}$, $B^e_{11} +\imath C_{11}= B^e_{11}(I+
\imath J) $. It is easily seen that $I+ \imath J$ is invertible,
so is $B^e_{11}$. To show the invertibility of $B^e_{22}$, it is
noticed that $B^e_{22}= \wt B^e_{22} R$, where $\wt B^e_{22}$ is
similar to  $B^e_{11}$ defined in (\ref{eqn:B_11}) with $\eta_i$
replaced by $\eta^+_i$ and $R = \mbox{diag}\big(M(\beta \omega_0),
M(\beta\omega_1), \cdots, M(\beta\omega_{\frac{m}{2}}) \big)$
where $M(\cdot)$ is given in (\ref{eqn:M_matrix}). Clearly $R$ is
invertible for all $\beta$, and it can be proved in the similar
way as for $B^e_{11}$ that $\wt B^e_{22}$ is nonsingular. Hence,
$B^e_{22}$ is invertible for all $\beta$. Consequently $\det(B^e)
=\det(B^e_{11}) \det(B^e_{22}) + O(e^{-m\beta}) \ne 0$ for all
$\beta$ sufficiently large. In addition, since each entry of the
adjoint of $B^e$ is bounded, we deduce that $(B^e)^{-1} = \frac{
\mbox{adj}(B^e) }{\det(B^e)}$ is bounded and the upper bound
depends on $m$ only, where $\mbox{adj}(B^e)$ stands for the
adjoint of $B^e$. Furthermore, letting $\kappa=\max_k(|c_k|,
|d_k|)$, where $c_k, d_k$ are the coefficients in the kernel $L$,
and $\varrho=\min\{ \mu_k, \, k=0, \cdots, \frac{m}{2}-1\}$, we
have, for $t_*=0$ or $1$,
\[
    \frac{ \big| F^{(j)}_0 (t_*) \big| }{ \beta^j } \, \le \, 2m
    \kappa \int^1_0 \beta e^{-\beta \varrho \tau} d\tau \, \| G\| \le 2m
    \kappa/\varrho  \, \| G \|, \ \ \ \ \forall \ j=1, \cdots, m-1.
\]
As a result, the equation $B^e x = \bf v$ has a unique solution
$\bf a$ that satisfies the desired bound.
\end{proof}

Consider an odd $m$. The homogeneous ODE: $ F^{(2m)} - \beta^{2m}
F =0$ has the following $2m$ (linearly independent) solutions:
\[
 e^{\pm\beta t}, \ \  e^{-\beta\mu_k t} \cos(\beta\omega_k t), \ \ \ e^{-\beta\mu_k t} \sin(\beta\omega_k
  t), \ \ \ \ e^{-\beta\mu_k(1- t)} \cos(\beta\omega_k t), \ \ \ e^{-\beta\mu_k (1-t)}
  \sin(\beta\omega_k t),
\]
where $k=1, \cdots, \frac{m-1}{2}$ and $\mu_k, \omega_k>0$ for the
above $k$. The solution to ODE (\ref{eqn:govern}) subject to the
boundary conditions can be written as
\begin{equation} \label{eqn:F_odd}
   F(t) \, = \, \underbrace{ \int^1_{0} P(|t-s|) G(s) ds }_{F_0(t)} + J(t),
\end{equation}
where
\begin{eqnarray} \label{eqn:J_t2}
J(t)  &=& a_0 e^{-\beta t} + a^+_0  e^{-\beta (1-t)} + \sum^{\frac{m-1}{2}}_{k=1} \Big\{ e^{-\beta\mu_k t}\big[ a_k
    \cos(\beta\omega_k t) + b_k \sin(\beta\omega_k t) \big] \nonumber \\
     && ~~~~~~~~~~~~~~~+ \ e^{-\beta\mu_k(1- t)}
   \big[ a^+_k \cos(\beta\omega_k t) +  b^+_k \sin(\beta\omega_k
   t) \big] \Big\},
   \end{eqnarray}
and the coefficients $a_k, b_k, a^+_k, b^+_k$ are to be
determined from the boundary conditions, and the kernel $P$ is
given in (\ref{eqn:P_odd}). Let
\begin{equation}\label{equ:b}
 {\bf b} \, = \,
 \Big(a_0, a_1, b_1, \cdots, a_{\frac{m-1}{2}}, b_{\frac{m-1}{2}}, a^+_0, a^+_1, b^+_1, \cdots,
    a^+_{\frac{m-1}{2}}, b^+_{\frac{m-1}{2}} \Big)^T
\end{equation}
be the coefficient vector. Similar to the case where $m$ is even,
we obtain the linear equation $B^o {\bf b} = \bf v$, where ${\bf
v}^{T}=[{\bf v_0}, {\bf v_1}]$
 and
\[
  B^o \, = \, \begin{bmatrix} B^o_{11} & B^o_{12} \\  B^o_{21} &
     B^o_{22} \end{bmatrix}.
\]
Here the matrix blocks $B^o_{ij} \in \mathbb R^{m\times m}$ are
obtained via the similar technique in
Section~\ref{sect:kernel_odd} as
\begin{eqnarray} \label{eqn:Bo_11}
 \lefteqn{ B^o_{11} \,  = \, } \nonumber \\
 & {\small \begin{bmatrix} 1 & 1 & 0 & \cdots & \cdots & 1 & 0 \\
               -1 & \cos(\gamma_1) & \sin(\gamma_1) &  \cdots & \cdots & \cos(\gamma_{\frac{m-1}{2}}) &
               \sin(\gamma_{\frac{m-1}{2}}) \\
                \vdots & \vdots &  \vdots & & & \vdots & \vdots \\
               (-1)^{m-1} & \cos((m-1)\gamma_1) & \sin((m-1)\gamma_1) &  \cdots & \cdots & \cos((m-1)\gamma_{\frac{m-1}{2}}) &
               \sin((m-1)\gamma_{\frac{m-1}{2} })
              \end{bmatrix}, } \nonumber \\
\end{eqnarray}
where $\gamma_k = (\pi-\frac{k\pi}{m}), k=1, 2, \cdots,
\frac{m-1}{2}$, and
\begin{eqnarray*}
  B^o_{22} \, = \, \begin{bmatrix} 1 & \cos(\zeta_{1,0}) & \sin(\zeta_{1,0}) & \cdots & \cdots
       &   \cos(\zeta_{\frac{m-1}{2}, 0}) & \sin(\zeta_{\frac{m-1}{2}, 0}) \\
      1 & \cos(\zeta_{1, 1}) & \sin(\zeta_{1, 1}) & \cdots & \cdots
       &   \cos(\zeta_{\frac{m-1}{2}, 1}) & \sin(\zeta_{\frac{m-1}{2}, 1}) \\
        \vdots &  \vdots & \vdots & &   & \vdots & \vdots \\
       1 & \cos(\zeta_{1, m-1}) & \sin (\zeta_{1, m-1}) & \cdots & \cdots
       &   \cos(\zeta_{\frac{m-1}{2}, m-1}) & \sin (\zeta_{\frac{m-1}{2}, m-1})
 \end{bmatrix},
\end{eqnarray*}
where   $\gamma^+_k = \frac{k\pi}{m}, \, \zeta_{k, \ell} = \beta\,
\omega_k + \ell\gamma^+_k$ for all $k=1, 2, \cdots, \frac{m-1}{2},
\, \ell=0, 1, \cdots, m-1 $, and each entry of $B^o_{12}$ and
$B^o_{21}$ is of order $O(e^{-\beta})$. To show the invertibility
of $B^o_{11}$, we introduce $E_{11} \equiv \big[ 0 \ B^o_{11}(3) \
B^o_{11}(2) \ B^o_{11}(5) \ B^o_{11}(4) \ \cdots \ B^o_{11}(m) \
B^e_{11}(m-1) \big]$, where $B^o_{11}(k)$ denotes the $k$th column
of $B^o_{11}$. As before it can be shown that $B^o_{11}+ \imath
E_{11}$ is nonsingular and $B^o_{11}+ \imath E_{11}= B^o_{11} (I +
\imath K)$, where $K\equiv \mbox{diag}(0, \underbrace{ J_*,
\cdots, J_*}_{\frac{m-1}{2}-\mbox{copies}})$ and $J_*$ is the
$2\times 2$ matrix defined before. Since $I+\imath K$ is
nonsingular, so is $B^o_{11}$. Furthermore, by applying the
similar technique, we can show that $B^o_{22}$ is invertible for
all $\beta$. This thus implies that for all $\beta$ sufficiently
large, $B^o$ is invertible and each entry of $(B^o)^{-1}$ is
bounded by a positive number depending on $m$ only. We summarize
the above discussions as follows:

\begin{lemma} \label{lem:banded_equation_odd}
  Given an odd $m$.
  There exist positive real numbers $\beta_*$ and $\varrho$, dependent on $m$ only, such
  that for all $\beta \ge \beta_*$, the coefficient vector $\bf b$ is unique
  and satisfies $\| \bf b \| \le \varrho \, \| \bf G \|$.
\end{lemma}

%
\section{Asymptotic properties of $P$-splines}\label{sec:asy}

To establish the asymptotic properties of the estimator, we first
represent $\hat F_m$ as the sum of the convolutions of $K(t, s)$
(defined in Section \ref{sec:gr}) with $G_m$ and a remainder term
that is of smaller order.

\begin{lemma}\label{lem:repF}
The $\hat F_m$ in (\ref{equ:ode}) can be represented as
\begin{equation}\label{equ:repF}
\hat{F}_m(t) =\int_0^1 K(s, t)G_m(s)ds + \int_0^1 K(s, t)\tilde R(s)ds + J(t),
\end{equation}
where $J(t)$ is given by (\ref{eqn:J_t1}) and (\ref{eqn:J_t2}) for even $m$ and odd $m$,  respectively.
The $\|\cdot\|_{\infty}$-norms of both coefficient vectors ${\bf a}$ in (\ref{eqn:a}) and  ${\bf b}$ in (\ref{equ:b})
are stochastically bounded, and $\|\tilde R\|  =O_p\Big( \big({\log K_n\over nK_n} \big)^{1/2} \Big)$.
\end{lemma}
\begin{proof}The representation of $\hat F_m$ in (\ref{equ:repF}) follows from the discussions in Section \ref{sec:gr}. The stochastic boundedness of the coefficient vectors is the direct applications of Lemma \ref{lem:banded_equation_even} and Lemma \ref{lem:banded_equation_odd}. Let $\bar y = {K_n\over n}X^Ty$ and $\lambda = \lambda^* K_n/n$. Claeskens et al. (2009) showed that $\|H^{-1}\|_{\infty} = O(1)$, where $H={K_n\over n}X^TX + \lambda D^T_mD_m$. Thus, $\hat b$ is stochastically bounded, so is $\hat f^{[p]}$. Let $\bar b$ solve  $(X^TX + \lambda^* D_m^TD_m)\bar{b} = X^T f$ and denote $\bar f(x) = \sum_{k=1}^{K_n+p}\bar b_k B_k^{(p)}(x)$.  We have
\begin{equation}
\|\hat f^{[p]} - \bar f\| \, \le \, \|\hat b - \bar b\|_{\infty}
\, \le \, \|H^{-1}\|_{\infty}\, \|\  \bar y - \mathbb E[\bar y]\
\|_{\infty} \, = \, O_p\Big(\sqrt{K_n\over n}\sqrt{2\log K_n}~
\Big).
\end{equation}
It is shown that  $\|\bar f - f\| =O(\alpha)$ if $p=m$.
The development of this result is a special case of Theorem
\ref{thm:main} in Section \ref{sec:asy}. Thus,
\begin{eqnarray*}
&& | \hat{F}_1(x)-G_1(x)+G_1(\kappa_{k_x+p})-\check{F}_1(\kappa_{k_x+p}) |
\\ &\le &  | \hat F_1(x) - \check F_1(x)| + | (G_1(\kappa_{k_x+p})-G_1(x))- (\Phi_1(\kappa_{k_x+p})-\Phi_1(x))|\\
&& + |(\Phi_1(\kappa_{k_x+p})-\Phi_1(x))-(\bar F_1(\kappa_{k_x+p})-\bar F_1(x))| + |(\bar F_1(\kappa_{k_x+p})-\bar F_1(x)) - (\check F_1(\kappa_{k_x+p})-\check F_1(x)) |\\
&\le & {2\over n}\|\hat f\| + O_p \Big({1\over \sqrt{nK_n}}\Big)+{pM_n\over n}\|\bar f - f\|+{pM_n\over n}\|\hat f - \bar f\|\\
& =& O_p\Big( \, {1\over n} \, \Big) + O_p\Big( \big({\log
K_n\over nK_n}\big)^{1/2} \Big) + O_p\Big({\alpha\over K_n} \Big).
\end{eqnarray*}
A similar rate can be obtained for $|R_{y,k_x+1}-R_{f,k_x+1}|$.
Given the admissible ranges of $K_n$ and $\alpha$ in next Corollary \ref{coro:refp}, $O_p(({\log K_n/ nK_n})^{1/2})$ is the dominating
term. Hence, the lemma follows.
\end{proof}

\begin{theorem}\label{thm:main}
If the true regression function is $2m$th order continuously
differentiable  with bounded $2m$th derivative, then the
$P$-spline estimator $\hat{f}^{[m]}$ can be written as
\begin{eqnarray}\label{equ:rep}
\hat{f}^{[m]}(t) &=& f(t) + (-1)^{m-1}\alpha f^{(2m)}(t) + o(\alpha)+
{1\over n}\sum_{i=1}^n K(t, t_i)\epsilon_i \\
&& ~~~~~~~~~~~~~~~~~~~~~~~~+ O_p\Big( \sqrt{\log K_n \over nK_n}~ \Big) \beta^m +
e^{-\beta t(1-t)}O_p( \beta^m ), \nonumber
\end{eqnarray}
uniformly in $\alpha$ and in $t\in (0,1)$.
\end{theorem}
\begin{proof} Taking the $m$th derivative of $\int_0^1 K(s, t)G_m(s)ds$, we
obtain
\begin{eqnarray*}
\int_0^1 {\partial^m K(t,s)\over \partial t^m}G_m(s)ds &=&\int_0^1
{\partial^m K(t,s)\over \partial t^m}\Phi_m(s)ds+\int_0^1
{\partial^m K(t,s)\over \partial t^m}\Big[G_m(s)-\tilde \Phi_m(s)\Big]ds\\
&&+\int_0^1 {\partial^m K(t,s)\over \partial t^m} \Big[\tilde
\Phi_m(s)-\Phi_m(s) \Big]ds.
\end{eqnarray*}
It is easy to show that
$$\int_0^1{\partial K(t,s)\over \partial
t}\Phi_m(s)ds=-\int_0^1 {\partial K(t,s)\over \partial
s}\Phi_m(s)ds = -\Phi_m(1)K(t, 1) + \int_0^1
K(t,s)\Phi_{m-1}(s)ds.$$ Therefore,
$$\int_0^1 {\partial^m K(t,s)\over \partial
t^m}\Phi_m(s)ds = -\sum_{j=0}^{m-1} \Phi_{m-j}(1) {\partial^{m-j}
\over \partial t^{(m-j)}}K(t, 1)+ \int_0^1 K(t,s)f(s)ds.$$ By
Equation (6.4) in Theorem 2.2 of Nychka (1995), we have
$$\int_0^1 K(t,s)f(s)ds = f(t) + (-1)^{m-1}\alpha f^{(2m)}(t) + o(\alpha).$$
Similarly, \begin{eqnarray*}\int_0^1 {\partial^m K(t,s)\over
\partial t^m}\Big[G_m(s)-\tilde \Phi_m(s) \Big]ds &=& -\sum_{j=0}^{m-1}
\Big[G_{m-j}(1)-\tilde \Phi_{m-j}(1) \Big] {\partial^{m-j} \over
\partial
t^{(m-j)}}K(t, 1)\\ &&+ \int_0^1 K(t,s)\Big[dG_1(s) - d\tilde \Phi_1(s) \Big],\\
&=& O_p\Big({\beta^m e^{-\beta(1-t)\over \sqrt{n}}}\Big) + {1\over
n}\sum_{i=1}^n K(t, t_i)\epsilon_i.
\end{eqnarray*}
Moreover,
$$\Big| \int_0^1 {\partial^m K(t,s)\over \partial
t^m}\Big[\tilde\Phi_m(s)-\Phi_m(s) \Big]ds \Big| \le
\|\tilde\Phi_m - \Phi_m\|~ ~\Big| \int_0^1 {\partial^m K(t,s)\over
\partial t^m}ds\Big|,$$ which is of order $O(1/n)\beta^m$.
Finally, in light of Lemma \ref{lem:repF}, $\|{d^m\over
dt^m}\int_0^1 K(s, t)\tilde R(s)ds\|$ is of order $(\log
K_n/nK_n)^{1/2} \beta^m$. It is easy to verify that the $m$th
derivative of $J(t)$ is of order $e^{-\beta t(1-t)}\beta^m$. This
completes the detail of the representation.
\end{proof}

%
%
\begin{remark} \rm
Theorem \ref{thm:main} indicates that the $P$-spline estimator is
approximately a kernel regression estimator. The equivalent kernel
is $K(t, s)$ given in Section \ref{sec:gr}, and $\alpha$ plays a
role similar to the bandwidth $h$. The asymptotic mean has the
bias $(-1)^{m-1}\alpha f^{(2m)}(x)$, which can be negligible if
$\alpha$ is reasonably small. On the other hand, $\alpha$ can not
be arbitrarily small as that will inflate the random component.
The admissible range for $\alpha$ given in Corollary
\ref{coro:asy} is a compromise between these two.
\end{remark}

\begin{coro}\label{coro:asy}
Let $\alpha$ satisfy $\alpha n^{2m/(4m+1)}\rightarrow 0$ and
$\alpha^{-(2m-1)/2m} \log K_n/K_n\rightarrow 0$. Suppose also that
the true regression function $f$ is $2m$th order continuously
differentiable with bounded $2m$th derivative. Then for $t\in
(0,1)$,
\begin{equation}\label{equ:asy1}
\sqrt{n\over \beta}~[\hat{f}^{[m]}(t) - f(t)] \rightarrow^d N\big(0,
\sigma_K^2(t)\big),
\end{equation}
where ${1\over \beta}\int_0^1 K^2(t,s)ds\rightarrow \sigma_K^2(t)$ as $n\rightarrow\infty$.
However, if $\alpha = c^{2m} n^{-{2m\over 4m+1}}$ for $c>0$, and let $K_n\sim n^\gamma$ with $\gamma>(2m-1)/(4m+1)$, then
\begin{equation}\label{equ:asy2}
n^{2m/ (4m+1)}~[\hat{f}^{[m]}(t) - f(t)] \rightarrow^d
N\Big((-1)^{m-1}c^{2m} f^{(2m)}(t), ~~{\sigma_K^2(t)\over c}~\Big).
\end{equation}
\end{coro}
\begin{proof} Let $\Pi(t) = {1\over n}\sum_{i=1}^n K(t, t_i)\epsilon_i$. For any fixed $t$, the Lindeberg-Levy central limit theorem gives
$$\sqrt{n\over \beta} ~\Pi(t)\rightarrow N(0, \sigma_K^2(t))$$
in distribution, where ${1\over \beta}\int_0^1 K^2(t,s)ds\rightarrow \sigma_K^2(t)$ as $n\rightarrow\infty$. If $\alpha$ satisfies $\alpha n^{2m/(4m+1)}\rightarrow 0$ and
$\alpha^{-(2m-1)/2m} \log K_n/K_n\rightarrow 0$, it is easy to see that the remainder terms in (\ref{equ:rep}) are $o_p(1)$. If $\alpha = c^{2m} n^{-{2m\over 4m+1}}$ for $c>0$, and $K_n\sim n^\gamma$ with $\gamma>(2m-1)/(4m+1)$, we have $\sqrt{n/\beta}\alpha = c^{2m+1/2}$ and $\sqrt{n/\beta} \sqrt{\log K_n \over nK_n} \beta^m\rightarrow 0$. The theorem follows.
\end{proof}

%
%
\begin{remark} \rm
The asymptotic results in Corollary \ref{coro:asy} provide
theoretical justification of the observation that the number of
knots is not important, as long as it is above some minimal level
(Ruppert, 2002). It is easy to find that the mean squared error of
the $P$-spline estimator is of order $n^{-4m/4m+1}$, which
achieves the optimal rate of convergence given in Stone (1982).
\end{remark}

In the following, we study the asymptotic property of
$\hat{f}^{[p]}(t) = \sum_{k=1}^{K_n+p}\hat{b}_k B_k^{[p]}(t)$ when
$p\neq m$. We first define a piecewise $m$th degree polynomial
$\tilde{f}^{[m]}$, where $\hat{f}^{[p]}$ and $\tilde{f}^{[m]}$
share the same set of spline coefficients. In particular, define
$\tilde{f}^{[m]}(t) = \sum_{k=1}^{K_n+m}\hat b_kB^{[m]}_k(t)$ if
$p>m$, or $\tilde{f}^{[m]}(t) = \sum_{k=1}^{K_n+p}\hat{b}_k
B^{[m]}_k(t)$ if $p< m$. Note that, if $p<m$, $\tilde{f}^{[m]}$ is
defined on $[0, 1-{m-p\over K_n}]$. Following the similar
discussion as above, we can establish the asymptotic distribution
for $\tilde{f}^{[m]}$ as in (\ref{equ:asy1}) and (\ref{equ:asy2}),
respectively, under different admissible ranges of $\alpha$ and
$K_n$.

\begin{lemma} \label{lem:diff}
For any $t\in (0,1)$, let $d= \lfloor K_n t \rfloor+1$. Let $\hat \gamma(t) =\hat{f}^{[p]}(t) - \tilde{f}^{[m]}(t)$. Then, if
$p>m$,
\begin{equation}\label{equ:rep0}
\hat \gamma(t)=
\sum_{q=m+1}^p\sum_{i=d+1}^{d+q} \Big({K_n\over
q}(t-\kappa_{i-q}) \Big)B_i^{[q-1]}(t) \sum_{l=1}^p a_{i+1-d, l}
K_n^{-l}{d^l\over dt^l} \hat f^{[p]}(t),
\end{equation}
and if $p<m$,
\begin{equation}\label{equ:rep1}
\hat \gamma(t) = -
\sum_{q=p+1}^m\sum_{i=d+1}^{d+m} \Big({K_n\over
q}(t-\kappa_{i-q}) \Big)B_i^{[q-1]}(t) \sum_{l=1}^m b_{i+1-d, l}
K_n^{-l}{d^l\over dt^l} \tilde f^{[m]}(t),
\end{equation}
where the coefficients $\{a_{ij}\}$ and $\{b_{ij}\}$ are constants.
\end{lemma}

\begin{proof}
The B-spline basis functions have the recurrence relationship such
that $$B_j^{[p]}(t) = {K_n\over
p}(t-\kappa_{j-p-1})B_{j-1}^{[p-1]}(t) + {K_n\over
p}(\kappa_{j}-t)B_{j}^{[p-1]}(t).$$ Let $f^{[p-1]}(t) =
\sum_{k=1}^{K_n+p-1}b_k B_k^{[p-1]}(t)$ with the same first
$(K_n+p-1)$ coefficients of $f^{[p]}$. For $x\in (\kappa_{d},
\kappa_{d+1})$, the difference between $f^{[p]}(t)$ and
$f^{[p-1]}(t)$ is given by
\begin{eqnarray}\label{equ:minus}
{f}^{[p]}(t) -{f}^{[p-1]}(t) &=& \sum_{i=d+1}^{d+p} \Big[
b_{i+1}{K_n\over p}(t-\kappa_{i-p}) +
b_i \big({K_n\over p}(\kappa_{i}-t)-1 \big) \Big] B_i^{[p-1]}(t) \nonumber \\
&=&\sum_{i=d+1}^{d+p} (b_{i+1}-b_i) \Big({K_n\over
p}(t-\kappa_{i-p}) \Big)B_i^{[p-1]}(t).
\end{eqnarray} From
(\ref{equ:minus}), if $p>m$,
\[
\hat{f}^{[p]}(t) = \tilde{f}^{[m]}(t) +
\sum_{q=m+1}^p\sum_{i=d+1}^{d+q} \Delta b_{i+1}\Big({K_n\over
q}(t-\kappa_{i-q}) \Big)B_i^{[q-1]}(t).
\]
From (\ref{equ:diff0}), we have $\Delta^l b_k = c_l^T(\Delta
b_{k-l+1}, \Delta b_{k-l+2}, \ldots, \Delta b_{k} )$, where $$c_l
= \Big[~(-1)^{l-1}{l-1\choose 0}, (-1)^{l-2}{l-1\choose 1},
\ldots, (-1)^{0}{l-1\choose l-1}~\Big]^T.$$ Combining this with
(\ref{equ:derivative}), it is easy to show that there exists
$C_d\in \mathbb{R}^{p\times p}$ such that
$$\Big[~\Delta b_{d+2}, \Delta b_{d+2}, \ldots, \Delta
b_{d+p+1}~\Big]^T = C_d ~\Big[~K_n^{-1}{d\over dt} f^{[p]}(t),
\ldots, K_n^{-p}{d^p\over dt^p} f^{[p]}(t)~\Big]^T.$$ Hence, we
can write \begin{equation}\Delta b_{d+k} = \sum_{l=1}^p a_{kl}
K_n^{-l}{d^l\over dx^l} f^{[p]}(t), \ k=2, \ldots,
p+1,\end{equation} which gives (\ref{equ:rep0}). (\ref{equ:rep1})
can be established similarly. Thus the lemma follows.
\end{proof}

\begin{coro}\label{coro:refp}
Suppose that $f$ is $2m$th order continuously differentiable with
bounded $2m$th derivative on $[0, 1]$. Let $\alpha$ satisfy
$\alpha n^{2m/(4m+1)}\rightarrow 0$ and $\alpha^{-(2m-1)/2m} \log
K_n/K_n\rightarrow 0$. Then, for $t\in (0,1)$,
\begin{equation}\label{equ:asy3}
\sqrt{n\over \beta}~[\hat{f}^{[p]}(t) - f(t) - \hat\gamma(t)] \rightarrow^d N\big(0,
\sigma_K^2(t)\big),
\end{equation}
where $\gamma(t)$ is given by (\ref{equ:rep0}) if $p<m$ or (\ref{equ:rep1}) if $p>m$.
However, if $\alpha = c^{2m} n^{-{2m\over 4m+1}}$ for $c>0$, and let $K_n\sim n^\gamma$ with $\gamma>(2m-1)/(4m+1)$, then
\begin{equation}\label{equ:asy4}
n^{2m/ (4m+1)}~[\hat{f}^{[p]}(t) - f(t)-\hat \gamma(t)] \rightarrow^d
N\Big((-1)^{m-1}c^{2m} f^{(2m)}(t), ~~{\sigma_K^2(t)\over c}~\Big).
\end{equation}
\end{coro}

\begin{remark} \rm When $p$ is not equal to $m$,
the asymptotic bias has an additional term $\hat \gamma(t)$, which
is of order $O_p(1/K_n)$. When $K_n$ grows sufficiently fast with
respect to $n$, this term is asymptotically negligible.
\end{remark}

%
\section{The equivalent kernels near boundary}
\label{sec:boundary}

The approximation of the equivalent kernel $K(t, s)$ deteriorates
when $t$ is near the boundary points of the design set. In this
section, we derive an explicit formula for the equivalent kernel
when $t$ is close to the boundary. We discuss the case when $t$ is
close to $0$ only; the case when $t$ is close to $1$ follows from
the similar argument and thus is omitted.

%

Consider an even $m$ first. It follows from the closed-form
expressions (\ref{eqn:F_t}) and (\ref{eqn:J_t1}) for $\hat F_m(t)$
that for $t \in[0, 1]$ sufficiently small, the $m$-th derivative
of $\sum^{\frac{m}{2}-1}_{k=0} e^{-\beta\mu_k(1- t)} \big[ a^+_k
\cos(\beta\omega_k t) +  b^+_k \sin(\beta\omega_k
 t) \big]$ is of order $O_p( \beta^m e^{-\beta})$. Hence, we only
consider
\begin{equation}\label{eqn:Ft}
  \wt F(t) \equiv  F_0(t) + \sum^{\frac{m}{2}-1}_{k=0}  e^{-\beta\mu_k t}\big[ a_k
    \cos(\beta\omega_k t) + b_k \sin(\beta\omega_k t) \big].\end{equation}
In the subsequent, we shall express the coefficients $a_k, b_k$ in
terms of $F_0(0)$ and its derivatives. This will eventually lead
to an explicit expression for the kernel.

In view of (\ref{eqn:L_even}), we have
\begin{equation} \label{eqn:F_0_d}
    F^{(j)}_0(0) = \int^1_{0} \frac{\partial L^{(j)}(|s-t|)}{\partial \, t^j} \Big|_{t=0} G(s)
     ds = (-1)^j \int^1_{0} \frac{\partial L^{(j)}(s)}{\partial \, s^j} G(s)
     ds, \ \ \ \forall \ j= 1, \cdots, 2m-1.
\end{equation}
Moreover, it follows from Section~\ref{sect:kernel_even} that
$\displaystyle L(s) = \beta \sum^{\frac{m}{2}-1}_{k=0} p_k(s) =
\beta
\sum^{\frac{m}{2}-1}_{k=0} \begin{bmatrix} \, 1 &  0 \end{bmatrix} \begin{bmatrix} p_k(s) \\
q_k(s) \end{bmatrix}$, where
\[
   \begin{bmatrix} p_k(s) \\ q_k(s) \end{bmatrix} = e^{-\beta \mu_k s} S(\omega_k \beta
  s ) \begin{bmatrix} c_k \\ d_k \end{bmatrix},
\]
and $S( \cdot) \in \mbox{SO}(2)$ is given by
$
  S(\cdot ) \, = \, \begin{bmatrix} \cos(\cdot) & \sin(\cdot) \\
      -\sin(\cdot) & \cos(\cdot) \end{bmatrix}.
$
In light of (\ref{eqn:p_q}), we have
\[
   \begin{bmatrix} p^{(j)}_k(s) \\
    q^{(j)}_k(s) \end{bmatrix} = (\beta A_k)^j \begin{bmatrix}
    p_k(s) \\ q_k(s) \end{bmatrix} = (\beta A_k)^j e^{-\beta \mu_k s} S(\omega_k \beta
  s ) \begin{bmatrix} c_k \\ d_k \end{bmatrix}.
\]
As a result, we obtain
\[
  \frac{\partial L^{(j)}(s)}{\partial \, s^j} \, = \,  \beta
\sum^{\frac{m}{2}-1}_{k=0} \begin{bmatrix} \, 1 &  0 \end{bmatrix} \begin{bmatrix} p^{(j)}_k(s) \\
q^{(j)}_k(s) \end{bmatrix} \, = \, \beta^{j+1}
  \sum^{\frac{m}{2}-1}_{k=0} (A_k)^j_{1 \bullet} e^{-\beta \mu_k s} S(\omega_k \beta
  s ) \begin{bmatrix} c_k \\ d_k \end{bmatrix},
\]
where $(A_k)^j_{1\bullet}$ denotes the first row of the $j$-th
power of $A_k$ given in (\ref{eqn:p_q}). This, along with
(\ref{eqn:F_0_d}), yields
\begin{eqnarray*}
   \frac{ F^{(j)}_0(0) }{\beta^j} & = & (-1)^j \int^1_{0} \Big( \sum^{\frac{m}{2}-1}_{k=0}
   (A_k)^j_{1 \bullet} \beta e^{-\beta \mu_k s} S(\omega_k \beta
  s ) \begin{bmatrix} c_k \\ d_k \end{bmatrix} \Big) G(s) ds \\
  & = &
 \sum^{\frac{m}{2}-1}_{k=0}  (-A_k)^j_{1 \bullet}  \int^1_{0} \Big(  \beta e^{-\beta \mu_k s} S(\omega_k \beta
  s ) \begin{bmatrix} c_k \\ d_k \end{bmatrix} \Big) G(s) ds.
\end{eqnarray*}
For notational simplicity, let $\wh B^e_{11}$ denote the inverse
of the matrix $B^e_{11}$ defined in (\ref{eqn:B_11}) and let
$${\bf p}_j = \sum^{\frac{m}{2}-1}_{k=0} (-A_k)^j_{1 \bullet} \int^1_{0}
    \Big(  \beta e^{-\beta \mu_k s} S(\omega_k \beta
  s ) \begin{bmatrix} c_k \\ d_k \end{bmatrix} \Big) G(s) ds,~~ j=0, \cdots, m-1,$$
and ${\bf p} = [{\bf p}_0, {\bf p}_1, \cdots, {\bf p}_{m-1}]^T\in \, \mathbb R^m$.
Therefore, it follows from the development in
Section~\ref{sect:boundary_conditions} that
\begin{equation} \label{eqn:a_and_b}
   \begin{bmatrix} a_0 \\ b_0 \\ \vdots \\ a_{\frac{m}{2}-1} \\
   b_{\frac{m}{2}-1} \end{bmatrix} \, = \, -(\wh B^e_{11} +
   O_p(e^{-\beta}) ) \, {\bf p} \, = \, -\wh B^e_{11} \, {\bf p} +
   O_p(e^{-\beta}).
\end{equation}
Returning to $\wt F^{(m)}(t)$ and using (\ref{eqn:a_and_b}), we
have, for $\beta\rightarrow \infty$,
\begin{eqnarray}
  \wt F^{(m)} (t) & = & F^{(m)}_0 (t) +  \beta^m
   \sum^{\frac{m}{2}-1}_{\ell=0} (A_\ell)^m_{1 \bullet} e^{-\beta \mu_\ell t} S(\omega_\ell \beta
   \, t ) \begin{bmatrix} a_\ell \\ b_\ell \end{bmatrix} \nonumber \\
   & = & \int^1_0  \frac{\partial L^{(m)}(|t-s|)}{\partial \, s^m} G(s)
     ds \, + \, \beta^m \, {\bf q}^T (t) \Big(-\wh B^e_{11} \Big) {\bf p}, \label{eqn:F_md_kernel_even}
\end{eqnarray}
where ${\bf q}(t) \equiv \big[{\bf q}_0(t), {\bf q}_1(t), \cdots,
{\bf q}_{\frac{m}{2}-1}(t) \big]^T \in \mathbb R^m$ with ${\bf
q}_\ell (t) \equiv (A_\ell)^m_{1 \bullet} e^{-\beta \mu_\ell t}
S(\omega_\ell \beta t )  \in \mathbb R^{1\times 2}$ for $\ell=0,
1, \cdots, \frac{m}{2}-1$.

To find the kernel in this case, particularly the kernel for the
second term, recall
\[
   \begin{bmatrix} p_k(s) \\ q_k(s) \end{bmatrix} = e^{-\beta \mu_k s} S(\omega_k \beta
  s ) \begin{bmatrix} c_k \\ d_k \end{bmatrix}.
\]
Therefore, the second term in (\ref{eqn:F_md_kernel_even}) becomes
$$\beta^m \, {\bf q}^T (t) \Big(-\wh B^e_{11} \Big) {\bf p} \, = \, \int_0^1 W(t,s) G(s)ds,$$
where
$$W(t,s) = \beta^m
   {\bf q}^T(t) \,
   \Big(-\wh  B^e_{11} \Big)\begin{bmatrix} \nu_0(s)\\ \nu_1(s) \\ \vdots \\ \nu_{m-1}(s)
  \end{bmatrix}$$ and $\nu_j(s) = \displaystyle \sum^{\frac{m}{2}-1}_{k=0}
   (-A_k)^j_{1 \bullet}
    \Big(  \beta \begin{bmatrix} p_k(s) \\ q_k(s) \end{bmatrix} \Big)$, $j=0, \ldots, m-1$.

Denote by $\ol p^{(r)}_k(s)$ and $\ol q^{(r)}_k(s)$ the $r$-th
order integrals of $p_k(s)$ and $q_k(s)$ respectively, namely,
\[
   \ol p^{(r)}_k(s)  \equiv \underbrace{\int \cdots
   \int}_{r-\mbox{copies}} p_k \, d\tau_1 \cdots d\tau_r,  \ \ \ \
   \ol q^{(r)}_k(s)  \equiv \underbrace{\int \cdots
   \int}_{r-\mbox{copies}} q_k \, d\tau_1 \cdots d\tau_r.
\]
In light of (\ref{eqn:p_q}), it is easy to verify that
\[
   \begin{bmatrix} \ol p^{(r)}_k(s) \\ \ol q^{(r)}_k(s) \end{bmatrix} = (\beta A_k)^{-r}
    \begin{bmatrix} p_k(s) \\ q_k(s) \end{bmatrix} = e^{-\beta \mu_k s} (\beta
    A_k)^{-r}  S(\omega_k \beta
  s ) \begin{bmatrix} c_k \\ d_k \end{bmatrix}.
\]
Using this and $\displaystyle \frac{\partial^m \ol
p^{(m)}_k(s)}{\partial s^m} = p_k(s), \frac{\partial^m \ol
q^{(m)}_k(s)}{\partial s^m} = q_k(s)$, we have
\begin{eqnarray*}
\nu_j(s) &=& \displaystyle \sum^{\frac{m}{2}-1}_{k=0}
   (-A_k)^j_{1 \bullet}
    \Big(  \beta \begin{bmatrix} \frac{\partial^m \ol
p^{(m)}_k(s)}{\partial s^m} \\ \frac{\partial^m \ol
q^{(m)}_k(s)}{\partial s^m} \end{bmatrix} \Big)= {\partial^m\over \partial s^m}  \displaystyle \sum^{\frac{m}{2}-1}_{k=0}
   (-A_k)^j_{1 \bullet}
    \Big(  \beta \begin{bmatrix} \ol p^{(m)}_k(s) \\  \ol q^{(m)}_k(s)
   \end{bmatrix} \Big)\\
 &=& {\partial^m\over \partial s^m}  \displaystyle \sum^{\frac{m}{2}-1}_{k=0}
   (-A_k)^j_{1 \bullet}
    \Big(  \beta  e^{-\beta \mu_k s} (\beta
    A_k)^{-m}  S(\omega_k \beta
  s ) \begin{bmatrix} c_k \\ d_k \end{bmatrix}   \Big).
\end{eqnarray*}
Therefore,
\begin{equation}
W(t,s) \, \equiv \,  {\partial^m\over \partial s^m} \Big( {\bf
q}^T(t) \,
   \big(-\wh  B^e_{11} \big)~ {\bf r}(s) \Big),
\end{equation}
where ${\bf r}(s) = [{\bf r}_0(s), {\bf r}_1(s), \ldots, {\bf
r}_{m-1}(s)]^T \in \mathbb R^m$, and
$${\bf r}_j(s) = \sum^{\frac{m}{2}-1}_{k=0}  (-A_k)^{-(m-j)}_{1 \bullet}
    \beta e^{-\beta \mu_k s} S(\omega_k \beta
  s ) \begin{bmatrix} c_k \\ d_k \end{bmatrix}, ~~~~j=0, \ldots, m-1.$$
Here the coefficients $c_k, d_k$ satisfy the linear equation
(\ref{eqn:coefficient_even}). Finally, we obtain the equivalent
kernel for $t \ge 0$ sufficiently small (when $\beta \rightarrow
\infty$) as
\begin{equation}
K_b(t,s) =  L(|t-s|) + {\bf q}^T(t) \,
   \big(-\wh  B^e_{11} \big)~ {\bf r}(s).
\end{equation}

\begin{figure}
\begin{center}
\resizebox{6in}{2in}{\includegraphics{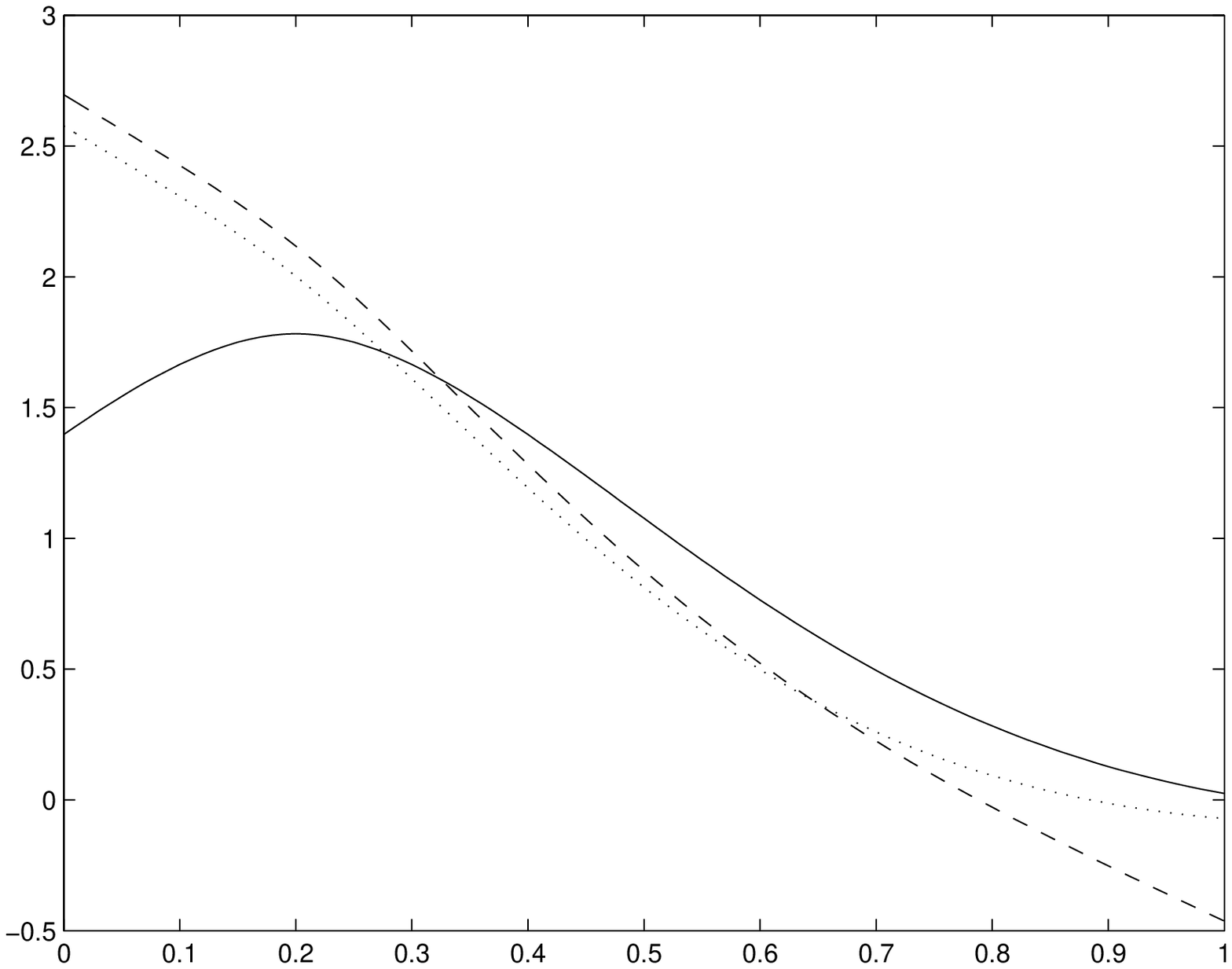}\hspace{0.2in}\includegraphics{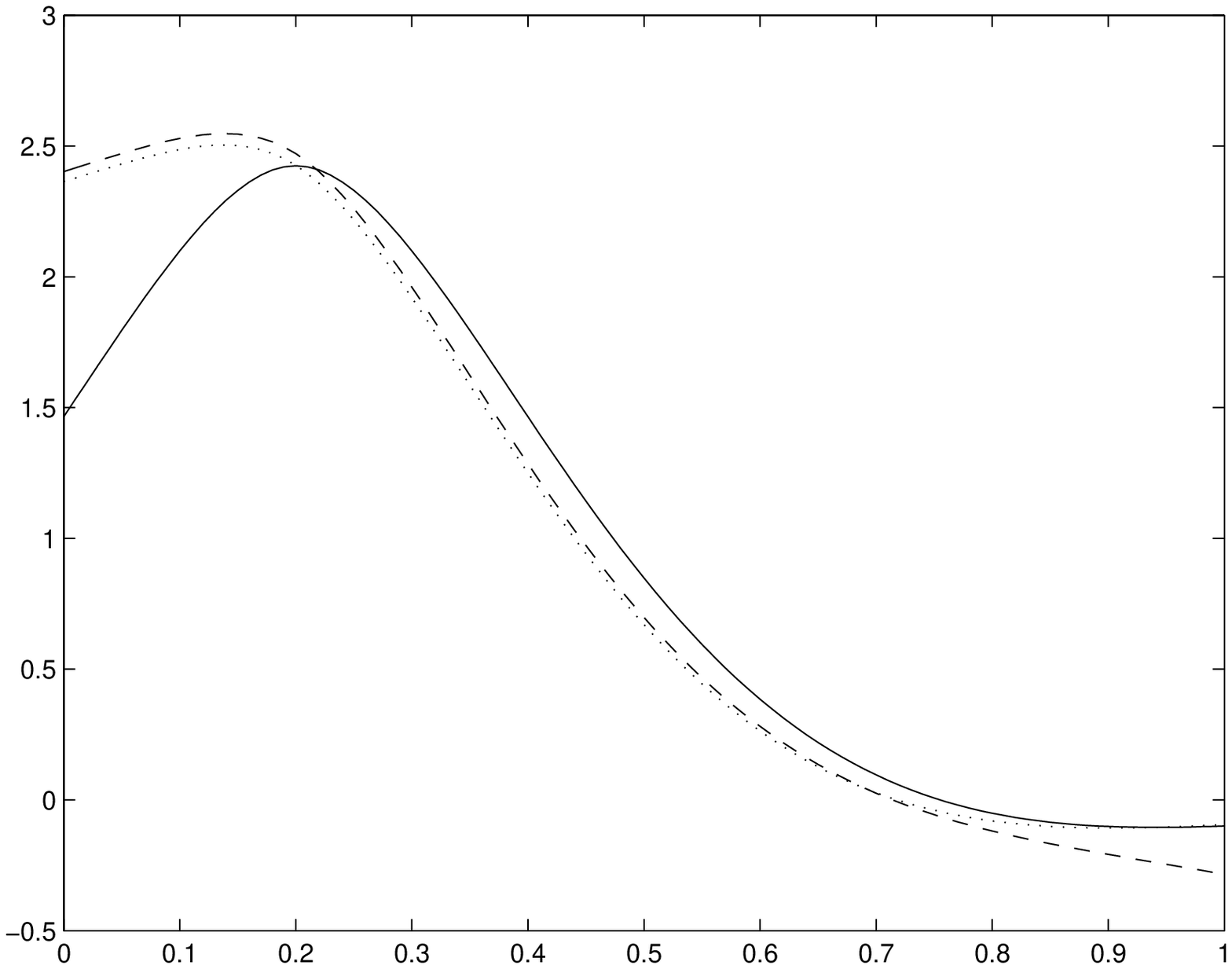}}
\resizebox{6in}{2in}{\includegraphics{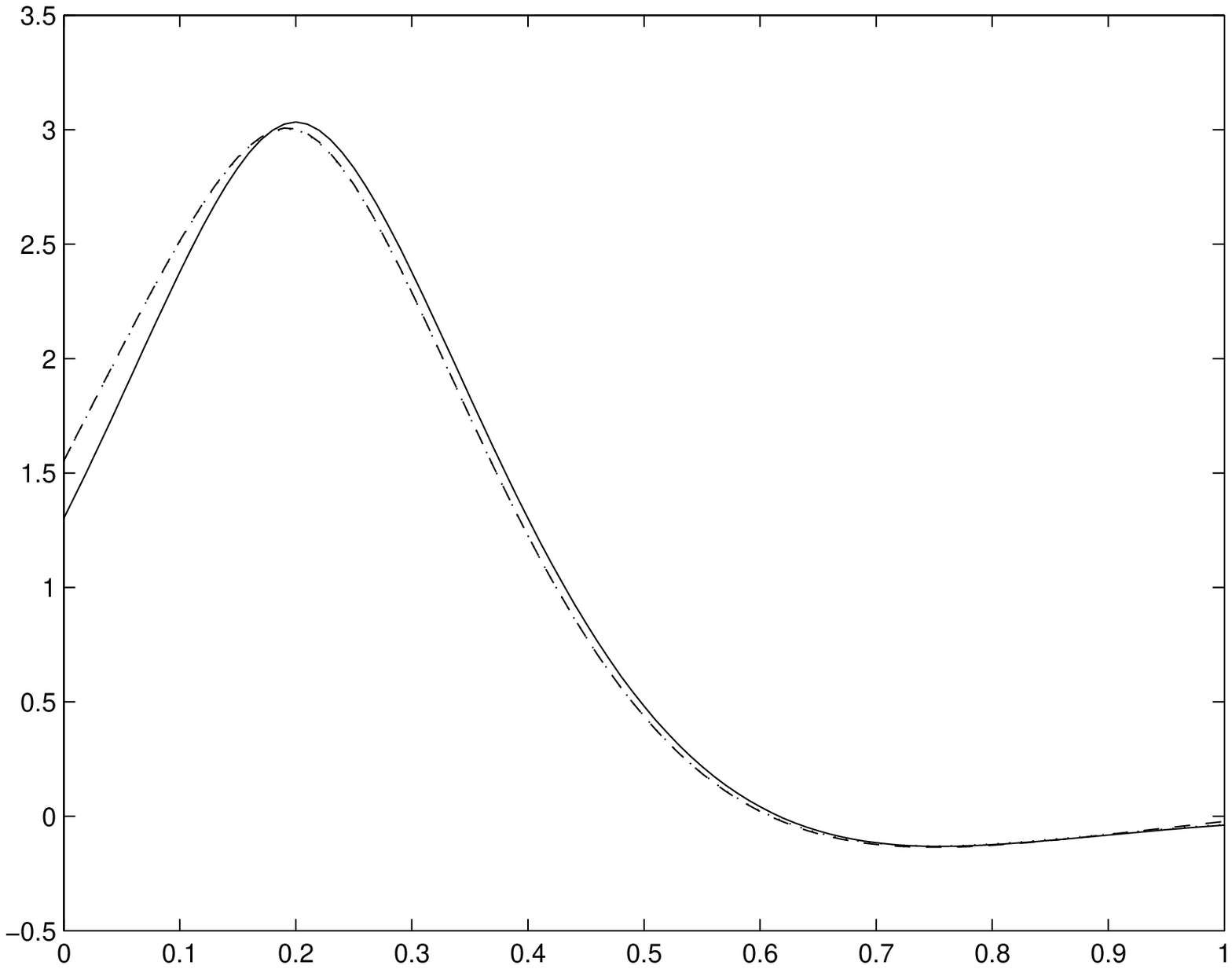}\hspace{0.2in}\includegraphics{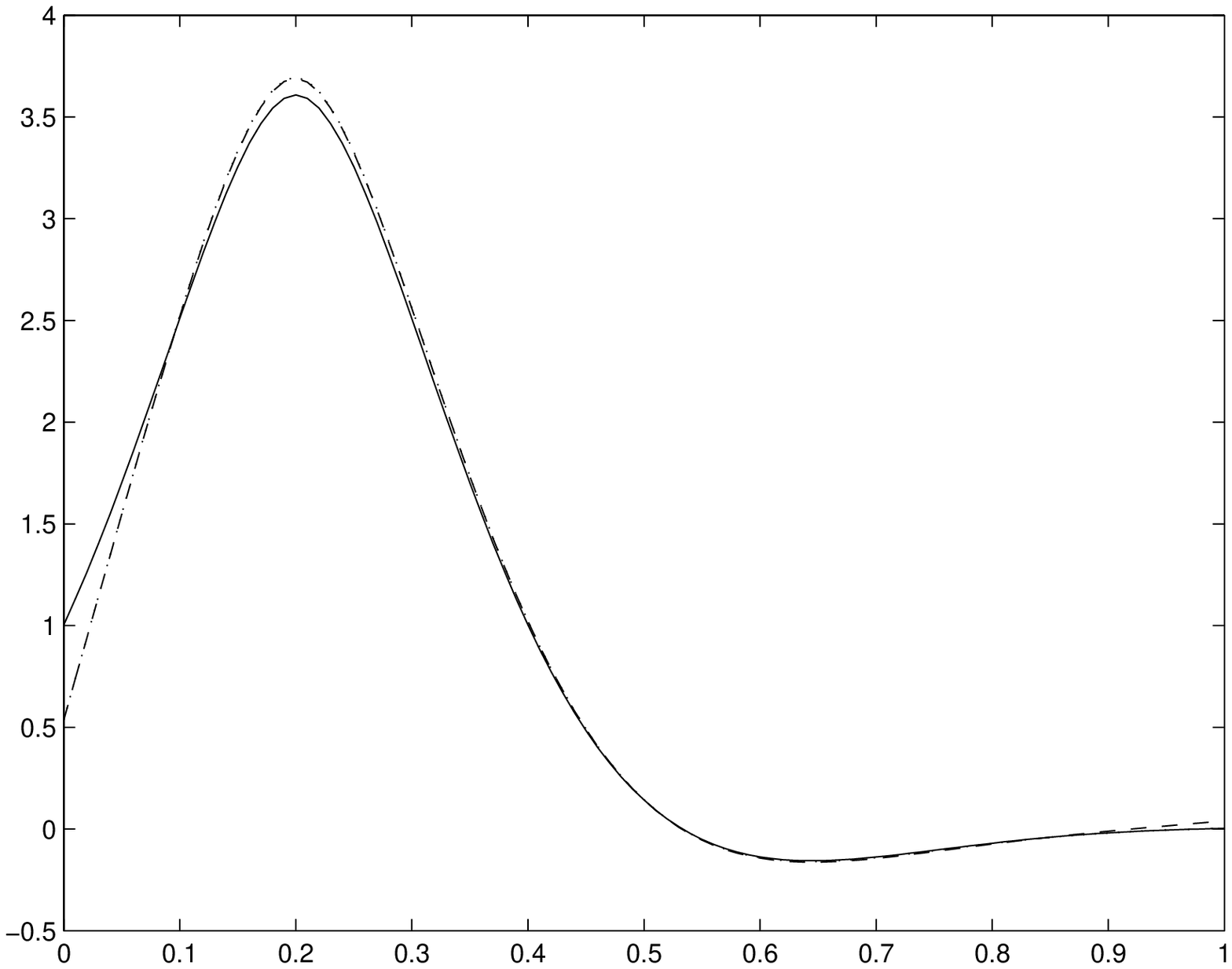}}
\end{center} \caption{The non-boundary kernel (solid), the finite-sample kernel (dashed)
and asymptotic boundary kernel (dotted) for $m = 2$ and for (a)
$\beta = 4$, (b) $\beta = 6$, (c) $\beta = 8$, (d) $\beta = 10$.
The kernels are for estimation at $x = 0.2$. }\label{fig:boundary}
\end{figure}

\begin{example} \label{example:kernel_boundary_m=2} \rm
  As an illustration, we derive the closed-form expression of the kernel near the boundary $t=0$
  for $m=2$
  and compare it with the boundary kernel established by Silverman (1984) for the smoothing splines. Since $c_0=d_0=
  \frac{1}{2\sqrt{2}}$,
  \[
     A_0 \,  = \, \begin{bmatrix} \cos(\frac{3\pi}{4}) &
     \sin(\frac{3\pi}{4}) \\ - \sin(\frac{3\pi}{4}) &
     \cos(\frac{3\pi}{4}) \end{bmatrix}, \ \ \ \ S(\omega_0 \beta t)
     \, = \, \begin{bmatrix} \cos(\frac{\beta t}{\sqrt{2}}) &
     \sin(\frac{\beta t}{\sqrt{2}}) \\ - \sin(\frac{\beta t}{\sqrt{2}}) &
     \cos(\frac{\beta t}{\sqrt{2}}) \end{bmatrix}, \ \ \ \
     \wh B^e_{11} = \begin{bmatrix} 1 & 0 \\ 1 & \sqrt{2}
     \end{bmatrix},
  \]
  we have
 \begin{equation*}
  {\bf q}(t) = e^{-\frac{\beta t}{\sqrt{2}}}
      \begin{bmatrix} \sin(\frac{\beta t}{\sqrt{2}}) \\ - \cos(\frac{\beta t}{\sqrt{2}})
       \end{bmatrix},~~~~~
   {\bf r}(t) =  \frac{\beta}{2\sqrt{2}} \, e^{-\frac{\beta s}{\sqrt{2}}}
     \begin{bmatrix} -\sin(\frac{\beta s}{\sqrt{2}}) + \cos(\frac{\beta s}{\sqrt{2}})
       \\ \sqrt{2} \cos(\frac{\beta s}{\sqrt{2}})
       \end{bmatrix}.
 \end{equation*}
Hence, the equivalent kernel near the boundary $t=0$ is
  \begin{eqnarray} \label{eqn:boundary_ker_even}
  && K_b(t,s)=  L(|t-s|) + {\bf q}^T(t) \,
   \Big(-\wh  B^e_{11} \Big)~ {\bf r}(s)  \\
 &=&  L(|t-s|) \, + \,   \frac{\beta}{2\sqrt{2}} \, e^{-\frac{\beta }{\sqrt{2}}(t+s) }
 \Big[   \cos\big( \frac{\beta }{\sqrt{2}}(t-s)
  \big) + 2 \cos\big(\frac{\beta}{\sqrt{2}} t \big) \cos\big(\frac{\beta}{\sqrt{2}} s
  \big) - \sin\big( \frac{\beta }{\sqrt{2}}(t+s) \big) \Big]. \notag
  \end{eqnarray}
 When $t=0$, since $L(|t-s|) = \frac{\beta}{2\sqrt{2}} \, e^{-\frac{\beta }{\sqrt{2}} s }
   \big[  \cos(\frac{\beta}{\sqrt{2}} s ) + \sin(\frac{\beta}{\sqrt{2}} s )
   \big]$, the boundary kernel becomes
   $\sqrt{2} \, e^{-\frac{\beta }{\sqrt{2}} s }
   \cos(\frac{\beta}{\sqrt{2}} s ), \ s\in[0, 1]$. It is interesting to notice that the boundary kernel in (\ref{eqn:boundary_ker_even}) agrees with that obtained by Silverman (1984). Figure \ref{fig:boundary} displays the non-boundary kernel, boundary kernel, and the finite sample kernel when we estimate $x=0.2$ with different choices of $\beta$, where the finite sample kernel is obtained by incorporating the terms containing $e^{-\beta\mu_k(1-t)}$ ignored
in (\ref{eqn:Ft}). Indeed, this kernel is given by
 \begin{eqnarray*}
 &   L(|t-s|) \, + \,   \frac{\beta}{2\sqrt{2}} \, e^{-\frac{\beta }{\sqrt{2}}(t+s) }
 \Big[   \cos\big( \frac{\beta }{\sqrt{2}}(t-s)
  \big) + 2 \cos\big(\frac{\beta}{\sqrt{2}} t \big) \cos\big(\frac{\beta}{\sqrt{2}} s
  \big) - \sin\big( \frac{\beta }{\sqrt{2}}(t+s) \big) \Big]
  \notag \\
 & + \frac{\beta}{2} \, e^{-\frac{\beta }{\sqrt{2}}(2-t-s) }
 \Big\{   \cos\big( \frac{\beta(1-t) }{\sqrt{2}}+\frac{\pi}{4}
  \big) \Big[\cos(\frac{\beta (1-s)}{\sqrt{2}}) -\sin(\frac{\beta (1-s)}{\sqrt{2}})
  \Big]
    + \sqrt{2} \cos(\frac{ \beta(1-t)}{\sqrt 2}) \cos(\frac{ \beta(1-s)}{\sqrt 2})
  \Big\}. \notag
  \end{eqnarray*}
  There are a good agreement between the finite-sample and asymptotic
 kernels when $\beta = 6$ and an excellent agreement when $\beta=10$.
\end{example}

%
The development of the boundary kernel for an odd $m$ is similar
and we omit the details here. For notational simplicity, let $\wh
B^o_{11}$ denote the inverse of the matrix $B^o_{11}$ defined in
(\ref{eqn:Bo_11}). we obtain the equivalent kernel for $t \ge 0$
sufficiently small (when $\beta \rightarrow \infty$) as
\begin{equation} \label{eqn:kernel_t=0_odd}
K_b(t,s) = P(|t-s|) + {\bf q}^T(t) \Big(-\wh
  B^o_{11} \Big) {\bf r}(s),
\end{equation}
where ${\bf q}(t)=\big[ (-1)^m e^{-\beta t}, {\bf q}_1(t), \ldots,
{\bf q}_{{m-1\over 2}}(t) \big]^T \in \mathbb R^m$ and ${\bf r}(s)
= [{\bf r}_0(s), {\bf r}_1(s), \ldots, {\bf r}_{m-1}(s)]^T \in
\mathbb R^m$, and
\begin{eqnarray*}
{\bf q}_\ell(t) &=&  (A_\ell)^m_{1 \bullet} e^{-\beta \mu_\ell t}
S(\omega_\ell \beta
  t )  \in \mathbb R^{1\times 2}, ~~~~ \ell =1, \ldots, {m-1\over 2},\\
 {\bf r}_j(s) &=&  \displaystyle (-1)^m c_0 \beta e^{-\beta s} + \sum^{\frac{m-1}{2}}_{k=1}  (-A_k)^{-(m-j)}_{1 \bullet}
    \beta e^{-\beta \mu_k s} S(\omega_k \beta
  s ) \begin{bmatrix} c_k \\ d_k \end{bmatrix}, ~~~ j=0, \ldots,
  m-1,
\end{eqnarray*}
where the coefficients $c_k, d_k$ satisfy the linear equation
(\ref{eqn:coefficient_odd}).

%
\section{Extensions to unequally spaced data and multivariate smoothing}\label{sec:diss}

We have so far focused on the equally spaced design case and
equally spaced knots. When the design is not equally spaced, one
can use the ideas of Stute (1984) and Li and Ruppert (2008). In
specific, assume that $x_i$'s are in $(a, b)$. Find a smoothing
monotone function $\Upsilon$ such that $\Upsilon(x_i)=i/n$ from $(a, b)$ to
$(0,1)$. We use the $P$-spline smoothing to fit $(i/n, y_i)$, and
thus the regression function is give by $f\circ \Upsilon^{-1}$. We place
knots at sample quantiles so that there are equal numbers of data
points between consecutive knots.

The univariate $P$-splines can be naturally extended to multivariate $P$-splines (Marx and Eilers, 2005). The asymptotic properties can be studied along the same line. Consider the problem of estimating the $\nu$ dimensional function $f(t_1, \ldots, t_\nu)$ from noisy observations $y_i = f(t_{1i}, \ldots, t_{\nu i})+ \epsilon_i$, $i =1,\ldots, n$. The $P$-spline model approximates $f$ by
$$f(t_1, \ldots, t_\nu) = \sum_{k_1=1}^{K_{1n}+p_1}\cdots \sum_{k_\nu=1}^{K_{\nu n}+p_\nu} b_{k_1, \ldots, k_\nu}B_{k_1}^{[p_1]}(t_1) \cdots B_{k_\nu}^{[p_\nu]}(t_\nu).$$
The spline coefficient $\hat{b}$ subject to the difference penalty are chosen to minimize
\begin{eqnarray*}
&&\sum_{i=1}^n \big[ y_i - \sum_{k_1=1}^{K_{1n}+p_1}\cdots \sum_{k_d=1}^{K_{dn}+p_d} b_{k_1, \ldots, k_d}B_{k_1}^{[p_1]}(t_{1i}) \cdots B_{k_d}^{[p_d]}(t_{di}) \big]^2 \\
&&~~~~~~~~~~~~~~~~~~~~~~~~~~~~~~~~~~ + \lambda^* \sum_{k_1=m_1+1}^{K_{1n}+p_1}\cdots \sum_{k_d=m_d+1}^{K_{dn}+p_d}\big[\Delta^{m_1, m_2, \ldots, m_d}b_{k_1, k_2,\ldots, k_d}\big]^2,
\end{eqnarray*}
where the difference operator for $d$ dimensional case is defined as follows:
\begin{eqnarray*}
\Delta^{0, \ldots, 0}b_{k_1, \ldots, k_\nu} &=& b_{k_1, \ldots, k_\nu}, ~~ k_1=1, \ldots, K_{1n}+p_1, \ldots, k_\nu=1,\ldots, K_{\nu n}+p_\nu,\\
\Delta^{m_1, m_2, \ldots, m_\nu}b_{k_1, k_2,\ldots, k_\nu} &=& \Delta^{m_1-1, m_2, \ldots, m_\nu}b_{k_1, k_2,\ldots, k_\nu}-\Delta^{m_1-1, m_2, \ldots, m_\nu}b_{k_1-1, k_2,\ldots, k_\nu}\\
&=& \Delta^{m_1, m_2-1, \ldots, m_\nu}b_{k_1, k_2,\ldots, k_\nu}-\Delta^{m_1, m_2-1, \ldots, m_\nu}b_{k_1, k_2-1,\ldots, k_\nu}\\
&=& \cdots\\
&=& \Delta^{m_1, m_2, \ldots, m_\nu-1}b_{k_1, k_2,\ldots, k_\nu}-\Delta^{m_1, m_2, \ldots, m_\nu-1}b_{k_1-1, k_2,\ldots, k_\nu-1}.
\end{eqnarray*}
For example, consider a two dimensional difference operator when $k_1=1$ and $k_2=2$:
\begin{eqnarray*}
\Delta^{1,2}b_{ks} &=& \Delta^{0,2}b_{ks} - \Delta^{0,2}b_{k-1,s}\\
&=& [b_{ks}-2b_{k,s-1}+b_{k,s-2}]- [b_{k-1,s}-2b_{k-1,s-1}+b_{k-1,s-2}].
\end{eqnarray*}
Let $X$ be the $n\times \{\Pi_{j=1}^\nu(K_{jn}+p_j)\}$  matrix with $(i,j)$th entry equal to $B_{k_1}^{[p_1]}(t_{1i}) \cdots B_{k_\nu}^{[p_\nu]}(t_{di})$. Define $D$ as the $\{\Pi_{j=1}^\nu(K_{jn}+p_j-m_j)\}\times \{\Pi_{j=1}^\nu(K_{jn}+p_j)\}$ differencing matrix satisfying
$$D b = \left (\begin{array}{c}
\Delta^{m_1,\ldots,m_\nu}b_{m_1+1, \ldots, m_\nu+1}\\
\vdots \\
\Delta^{m_1,\ldots,m_\nu}b_{K_{1n}+p_1,\ldots, K_{\nu n}+p_\nu}\\
\end{array}
\right).$$
The optimality condition is given by
\begin{equation}
(X^TX + \lambda^* D^TD)\hat{b} = X^Ty.
\end{equation}
Note that $D = D_{m_1}\otimes D_{m_2} \otimes \cdots \otimes
D_{m_\nu}$ and $D^TD =  D_{m_1}^TD_{m_1}\otimes D_{m_2}^TD_{m_2}
\otimes \cdots \otimes D_{m_d}^TD_{m_d}$, where ``$\otimes$''
represents the Kronecker product.  We may go though the same
procedure as described in this paper. The multivariate $P$-spline
smoothing is asymptotically equivalent to kernel smoothing and the
equivalent kernel is the Green's function corresponding to the
partial differential equation (PDE):
\begin{equation}
(-1)^{m_1+\cdots+m_d} \alpha {\partial^{2m_1+\cdots +2m_d}\over \partial t_1^{2m_1}\cdots\partial t_d^{2m_d}}F(t_1,\ldots, t_d)
+F(t_1,\ldots, t_d) = G(t_1, \ldots, t_d),
\end{equation}
subject to the boundary conditions:
\begin{eqnarray*}
&&{\partial^{k_1+\cdots +k_d}\over \partial t_1^{k_1}\cdots\partial t_d^{k_d}}F(t_1,\ldots, t_d) = 0, \mbox{~if any }t_i=0, ~~k_i=0,\ldots,m_i-1,\\
&&{\partial^{k_1+\cdots +k_d}\over \partial t_1^{k_1}\cdots\partial t_d^{k_d}}F(t_1,\ldots, t_d) ={\partial^{k_1+\cdots +k_d}\over \partial t_1^{k_1}\cdots\partial t_d^{k_d}}G(t_1,\ldots, t_d), \mbox{~if any }t_i=1, ~~k_i=0,\ldots,m_i-1.
\end{eqnarray*}
Further study of this issue is beyond the scope of this paper and
shall be reported in a future publication.



\begin{thebibliography}{99}

\bibitem{cko} Claeskens, G., Krivobokova, T. and Opsomer, J. (2009). Asymptotic properties of penalized spline estimators. {\it Biometrika}, in print.

\bibitem{deBoor_book} De Boor, C. (2001) {\it A Pratical Guide to
Splines}. Springer.



\bibitem{ho} Hall, P. and Opsomer, J.D. (2005). Theory for penalised spline
regression. {\it Biometrika}, {\bf 92}, 105-118.



\bibitem{lr} Li, Y. and Ruppert, D. (2008). On the asymptotics of
penalized splines. {\it Biometrika}, {\bf 95}, 415-436.



\bibitem{mm} Mammen, E. (1991). Estimating a smoothing regression
function. {\it Annals of Statistics}, {\bf 19}, 724-740.


\bibitem{ME} Marx, B. and Eilers, P. (1996). Flexible smoothing with B-splines and penalties (with comments and rejoinder). {\it Statistical Science}, {\bf 11}, 89-121.

\bibitem{ME} Marx, B. and Eilers, P. (2005). Multidimensional penalized signal regression. {\it Technometrics}, {\bf 47} 13-22.

\bibitem{me}Messer, K. (1991). A comparison of a spline estimate to
its equivelent kernel estimate. {\it Annals of Statistics}, {\bf
19}, 817-829.


\bibitem{ny} Nychka, D. (1995). Splines as local smoothers. {\it Annals of
Statistics}, {\bf 23}, 1175-1197.

\bibitem{os} O'Sullivan, F. (1986). A statistical perspective on ill-posed inverse problems (with Discussion), {\it Statistical Science}, {\bf 1}, 505-527.



\bibitem{pw} Pal, J. and Woodroofe, M. (2007). Large sample
properties of shape restricted regression estimators with
smoothness adjuctments. {\it Statistica Sinica}, {\bf 17},
1601-1616.


\bibitem{rr} Rice, J. and Rosenblatt, M. (1983). Smoothing
splines: regression, derivatives and deconvolution. {\it Annals of
Statistics}, {\bf 11}, 141-156.

\bibitem{r} Ruppert, D. (2002). Selecting the number of knots for
penalized splines. {\it Journal of Computational and Graphical
Statisitcs}, {\bf 11}, 735-757.

\bibitem{rc} Ruppert, D. and Carroll, R. (2000).
Spatially-adaptive penalities for spline fitting. {\it Australian
\& New Zealand Journal of Statistics}, {\bf 42}, 205-224.

\bibitem{rwc} Ruppert, D., Wand, M.P., and Carroll, R.J. (2003). {\it Semiparametric
Regression}. Cambridge: Cambridge University Press.





\bibitem{s} Silverman, B.W. (1984). Spline smoothing: the equivalent variable kernel
method. {\it Annals of Statistics}, {\bf 12}, 898-916.


\bibitem{SWang_ACC09}
Shen, J. and Wang, X. (2009). Estimation of shape constrained
functions in dynamical systems and its application to genetic
networks. Submitted to 2010 American Control Conference.

\bibitem{stone} Stone, C.J. (1982). Optimal rate of convergence for nonparametric regression. {\it
Annals of Statistics}, {\bf 10}, 1040-1053.

\bibitem{st} Stute, W. (1984). Asymptotic normality of nearest
neighbor regression function estimates. {\it Annals of
Statistics}, {\bf 12}, 917-926.

\bibitem{WangS_annal09}
Wang, X. and Shen, J. (2009). On the asymptotics of monotone
$P$-splines. Manuscript.






\end{thebibliography}
\end{document}